\documentclass[a4paper]{amsart}

\usepackage{a4wide}
\usepackage{bbm}
\usepackage{amssymb,amscd}
\usepackage{MnSymbol}

\usepackage{pstricks,epsfig}
\usepackage[totalwidth=14cm,totalheight=24cm]{geometry}
\usepackage{graphicx}
\usepackage{psfrag}
\usepackage{colortbl}
\usepackage[normalem]{ulem}

\usepackage[all]{xy}

\newcounter{notes}

\usepackage{mathrsfs,bbm}

\DeclareFontFamily{OT1}{pzc}{}
\DeclareFontShape{OT1}{pzc}{m}{it}{<-> s * [1.100] pzcmi7t}{}
\DeclareMathAlphabet{\mathpzc}{OT1}{pzc}{m}{it}

\newtheorem{cor}{Corollary}[section]
\newtheorem{theorem}[cor]{Theorem}
\newtheorem{prop}[cor]{Proposition}
\newtheorem{lemma}[cor]{Lemma}
\newtheorem{sublemma}[cor]{Sublemma}

\newtheorem{maintheorem}{Theorem}
\newenvironment{mainthm}[1]{\renewcommand{\themaintheorem}{#1}\begin{maintheorem}}{\end{maintheorem}}

\theoremstyle{definition}
\newtheorem{defi}[cor]{Definition}
\theoremstyle{remark}
\newtheorem{remark}[cor]{Remark}

\newtheorem{question}[cor]{Question}

\newtheorem*{notation}{Notation}

\def\IM{\mathfrak{I}}
\def\Hom{\mathrm{Hom}}
\def\End{\mathrm{End}}

\newcommand{\coh}{\mathcal{H}}
\newcommand{\cob}{\mathcal{B}}

\newcommand{\coc}{\mathcal{Z}}
\newcommand{\e}{\varepsilon}

\newcommand{\cT}{{\mathcal T}}
\newcommand{\cQ}{{\mathpzc Q}}

\newcommand{\cX}{{\mathpzc X}}
\def\bm#1{\text{\boldmath$#1$}}
\newcommand{\da}{\omega_h}

\def\hra{\hookrightarrow}
\def\Hod{\mathpzc{H}}

\newcommand{\F}{\mathsf{F}}
\newcommand{\Ae}{\mathsf{A}}
\newcommand{\Se}{\mathsf{S}}

\newcommand{\Mon}{\mathsf{Mon}}
\newcommand{\Data}{\mathpzc{D}}

\newcommand{\cod}{\mathrm{cod}}
\newcommand{\Cod}{\mathpzc{Cod}}
\newcommand{\cQC}{\mathpzc{Q}_{\,\CC}}
\newcommand{\Codtr}{\mathpzc{Cod}_{\!\! tr}}
\newcommand{\Id}{\mathbbm{1}}

\newcommand{\Imm}{\mathcal{I}}
\newcommand{\MImm}{\mathpzc{M}\mathcal{I}}
\newcommand{\AImm}{\widetilde{\Imm}}
\newcommand{\AMImm}{\widetilde{\MImm}}

\newcommand{\cML}{{\mathcal M\mathcal L}}

\newcommand{\cQF}{{\mathcal Q\mathcal F}}

\newcommand{\C}{{\mathbb C}}
\newcommand{\CC}{{\mathbb{C}}}

\newcommand{\HH}{{\mathbb H}}

\newcommand{\N}{{\mathbb N}}
\newcommand{\R}{{\mathbb R}}

\newcommand{\psl}{\mathfrak{sl}}
\newcommand{\SL}{\mathrm{SL}}
\newcommand{\PSL}{\mathrm{PSL}}

\newcommand{\Map}{\mathcal{C}}
\newcommand{\AMap}{\widetilde{\Map}}

\newcommand{\tr}{\mbox{\rm tr}}

\newcommand{\grad}{\operatorname{grad}}

\newcommand{\rar}{\rightarrow}
\newcommand{\lra}{\longrightarrow}

\def\rar{\rightarrow}
\def\pa{\partial}
\def\RR{\mathbb{R}}

\def\Hess{\mathrm{Hess}}
\def\Div{\mathrm{div}}

\def\ol#1{\overline{#1}}

\def\En{\mathbb{I}}

\def\Htype#1{\check{#1}}
\def\kill{\xi}
\def\Hkill{{\Htype{\kill}}}
\def\HX{{\Htype{X}}}
\def\HY{{\Htype{Y}}}
\def\dotr{\varsigma}

\newcommand{\EE}{\mathbb{E}}
\newcommand{\DD}{\mathbb{D}}



\begin{document}

\title{Minimizing immersions of a hyperbolic surface\\ in a hyperbolic $3$-manifold}
\author{Francesco Bonsante}
\address{Universit\`a degli Studi di Pavia\\
Via Ferrata, 1\\
27100 Pavia, Italy}
\email{francesco.bonsante@unipv.it}
\thanks{F.B.~was partially supported by the
FIRB 2010 project ``Low-dimensional geometry and topology'' (RBFR10GHHH 003) granted by MIUR}
\author{Gabriele Mondello}
\address{``Sapienza'' Universit\`a di Roma - Dipartimento di Matematica
``Guido Castelnuovo'' \\
piazzale Aldo Moro 5 \\
00185 Roma, Italy}
\email{mondello@mat.uniroma1.it}
\thanks{G.M.~was partially supported by the
FIRB 2010 project ``Low-dimensional geometry and topology'' (RBFR10GHHH 003) granted by MIUR and by GNSAGA 
national research group at INdAM}
\author{Jean-Marc Schlenker}
\thanks{JMS is partially supported by UL IRP grant NeoGeo and FNR grants INTER/ANR/15/11211745 and OPEN/16/11405402. JMS also acknowledge support from U.S. National Science Foundation grants DMS-1107452, 1107263, 1107367 ``RNMS: GEometric structures And Representation varieties'' (the GEAR Network).}
\address{Mathematics Research Unit, University of Luxembourg,
Maison du nombre, 6 avenue de la Fonte, L-4364 Esch-sur-Alzette, Luxembourg}
\email{jean-marc.schlenker@uni.lu}

\date{\today (v1.1)}

\begin{abstract}
Let $(S,h)$ be a closed hyperbolic surface and $M$ be a quasi-Fuchsian 3-manifold. We consider incompressible maps from $S$ to $M$ that are critical points of  an energy
functional $F$
 which is homogeneous of degree $1$. These ``minimizing'' maps are solutions of a non-linear elliptic equation, and reminiscent of harmonic maps -- but when the target is Fuchsian, minimizing maps are minimal Lagrangian diffeomorphisms to the totally geodesic surface in $M$. We prove the uniqueness of smooth minimizing maps 
 from $(S,h)$ to $M$
 in a given homotopy class. When $(S,h)$ is fixed, smooth minimizing maps from $(S,h)$ are described by a simple holomorphic data on $S$: a complex self-adjoint Codazzi tensor of determinant $1$. The space of admissible data is smooth and naturally equipped with a complex structure, for which the monodromy map taking a data to the holonomy representation of the image is holomorphic. Minimizing maps are in this way reminiscent of shear-bend coordinates, with the complexification of $F$ analoguous to the complex length.
\end{abstract}

\maketitle

\tableofcontents

\section{Introduction and results}

The aim of this paper is to study ``minimizing immersions'' of a compact hyperbolic
surface inside germs of hyperbolic 3-manifolds, which are defined as critical points
of a suitable $1$-homogeneous functional $F$.

On one hand, such immersions generalize the notion of minimal Lagrangian maps
between hyperbolic surfaces (which correspond to the target being a Fuchsian hyperbolic
3-manifold). On the other hand, as the target varies in the quasi-Fuchsian space,
a very natural complexification of $F$ can be used to define a holomorphic function
that looks like a ``smooth version'' of the complex length of a measured lamination.

\subsection{Minimal Lagrangian diffeomorphisms between hyperbolic surfaces}

Consider a compact, connected, oriented surface $S$ of genus at least two.

Given two hyperbolic metrics on $S$, 
a central problem in Teichm\"uller theory is to find the ``best''
diffeomorphism of $S$ isotopic to the identity. 
Usually such a diffeomorphism is the unique map
that minimizes a suitable functional defined in terms of
the two hyperbolic metrics.

One remarkable example is that of harmonic maps, which
play an important role in Teichm\"uller theory (see 
\cite{eells-sampson:harmonic} and \cite{wolf:teichmuller}).

Another example is given by minimal Lagrangian maps,
which are quite relevant both for Teichm\"uller theory and for 3-dimensional manifolds of constant curvature, see e.g. \cite{L5,L6,cyclic,cyclic2}.

\begin{defi}[Minimal Lagrangian maps]\label{def:minimal-lag}
A {\it minimal Lagrangian} map $m:(S,h)\to (S,h^\star)$
between hyperbolic surfaces is an
area-preserving diffeomorphism such that
its graph in $(S\times S, h\oplus h^\star)$ is minimal.
\end{defi}


Here we present another variational characterization of minimal Lagrangian maps between hyperbolic surfaces.

Let $f:(S,h)\to (S,h^\star)$ 
be a smooth\footnote{In this paper we use the word ``smooth'' to mean ``$C^\infty$''.} map between hyperbolic surfaces. 
There is a unique non-negative $h$-self-adjoint operator $b:TS\to TS$
such that $f^*h^\star(\bullet,\bullet)=h(b\bullet,b\bullet)$.
We define the functional $F:C^\infty(S,S)\to \RR$ 
on the space of $C^\infty$ maps from $S$ to $S$ as
\[
F(f):=\int_S \tr(b)\,\da
\]
where $\da$ is the area form on $S$ associated to $h$.

The following statement is almost implicit in some variational formulas in \cite{cyclic,cyclic2}, and also similar to results of Trapani and Valli \cite{trapani-valli}.

\begin{lemma}[Variational characterization of minimal Lagrangian maps]\label{lm:F}
Let $f:(S,h)\to (S,h^*)$ be a smooth diffeomorphism between
hyperbolic surfaces. 
Then $f$ is minimal Lagrangian if and only if it is a critical point of $F$. 
\end{lemma}

One of the key motivations here is to extend the notion of minimal Lagrangian diffeomorphism to smooth maps from a hyperbolic surface to a hyperbolic 3-manifold.

\subsection{Minimizing immersions of surfaces in 3-manifolds}
\label{ssc:minimizing}
Suppose now that the target surface is replaced by a hyperbolic 3-manifold $M$, which we assume to be complete and with
injectivity radius positively bounded from below.

In a given homotopy class $[f]$ of embeddings of $S$ into $M$
that induce an injective homomorphism $f_*:\pi_1(S)\rar\pi_1(M)$,
there is still a unique harmonic map from $(S,h)$ to $M$, but its relation to the moduli space of hyperbolic structures on $M$ is not direct --- for instance, the complex structure on this moduli is more readily visible if one fixes on $S$ a {\em measured lamination} rather than a metric, and considers shear-bend coordinates associated to it, see \cite{bonahon-toulouse}.

The analog of minimal Lagrangian maps for embeddings of a hyperbolic surface in a hyperbolic 3-manifold is not clear, if one follows 
Definition \ref{def:minimal-lag}. On the other hand, it is possible
to adapt the variational approach
suggested by Lemma \ref{lm:F}.\\

Let $(S,h)$ be a hyperbolic surface and $(M,h_M)$ be a hyperbolic 3-manifold
and let $f:S\to M$ be a smooth map.
Again, there exists a unique non-negative $h$-self-adjoint operator
$b:TS\to TS$ such that $f^*h_M(\bullet,\bullet)=h(b\bullet,b\bullet)$.
We define the functional $F:C^\infty(S,M)\to\RR$ as
\[
F(f):=\int_S \tr(b)\,\da~.
\]
Lemma \ref{lm:F} then suggests the following definition.

\begin{defi}[Minimizing maps]\label{df:minimizing}
A smooth map $f:S\to M$ from a hyperbolic surface
to a hyperbolic 3-manifold
is {\em minimizing} 
if $F$ achieves a local minimum at $f$.
\end{defi}

We will see that minimizing immersions have a number of pleasant properties. Given $(S,h)$, $M$ and the homotopy class 
$[f]$ of a map $f:S\to M$, there is at most one minimizing immersion from $(S,h)$ to $M$ in $[f]$. Moreover, the moduli space of minimizing immersions of $(S,h)$ in hyperbolic 3-manifolds has a complex structure, for which the map sending a minimizing immersion to the holonomy representation of the target manifold is holomorphic. 

\subsection{Definition and notations}

We now fix the background hyperbolic metric $h$ on $S$
and let $\tilde{h}$ be the pull-back of $h$ to the universal
cover $\widetilde{S}\rightarrow S$.

Rather than considering immersions of $S$ into a hyperbolic 3-dimensional manifold $M$, it is often more convenient to consider equivariant immersions of $\widetilde{S}$ into $\HH^3$, so that deformations of $M$ correspond to deformations of the representation.

\begin{defi}[Equivariant immersions]
An {\it immersion of $S$ in a germ of a hyperbolic $3$-manifold} is a couple $(\tilde{f},\rho)$, where  $\rho:\pi_1(S)\rightarrow \PSL_2(\CC)$ is a representation and $\tilde{f}:\widetilde{S}\rightarrow\HH^3$ is a $\rho$-equivariant smooth immersion.
The representation $\rho$ is called the {\it{monodromy}} of the map $\tilde{f}$.
\end{defi}

The set $\AImm$ of smooth immersions of $S$ in a germ of hyperbolic $3$-manifold is a subset of $\mathrm{Hom}(\pi_1(S),\PSL_2(\CC))\times C^\infty(\widetilde{S},\HH^3)$ and so it inherits a subspace topology.

Note that $\PSL_2(\CC)$ acts on $\AImm$ as $g\cdot (\tilde{f},\rho):=(g\circ\tilde{f},g\rho g^{-1})$.
We denote by $[\tilde{f},\rho]$ the orbit of $(\tilde{f},\rho)$ under this action and by $\Imm$ the quotient $\PSL_2(\CC)\backslash\AImm$, which is thus endowed with the quotient topology.

We also say that  the family of equivalence classes $[\tilde{f}_t,\rho_t]_{t\in I}$ is smooth if it can be represented by a smooth family of immersions.

\begin{remark}[Equivariant immersions and classes of immersions]
In order to explain the above definition, consider two immersions
$f_1:S\rar M_1$ and $f_2:S\rar M_2$ inside two (not necessarily compact) hyperbolic $3$-manifolds $M_1$ and $M_2$. 
We declare the two immersions equivalent if there exists a third immersion $f_3:S\rar M_3$ into a hyperbolic $3$-manifold $M_3$ and local isometries $i_1:M_3\rar M_1$ and $i_2:M_3\rar M_2$
such that $f_1=f_3\circ i_1$ and $f_2=f_3\circ i_2$.
Lifting such an  immersion $f:S\rar M$ to the universal covering $\widetilde{M}$ and composing with a developing map $\mathrm{dev}:\widetilde{M}\to\HH^3$, one gets an immersion
$\tilde{f}:\widetilde{S}\rar \HH^3$, which is equivariant under $\pi_1(S)$ that acts on $\widetilde{S}$ by deck transformations and on $\HH^3$ via the 
representation $\rho:\pi_1(S)\rar\PSL_2(\CC)$ obtained by composing the holonomy of $M$ with the homomorphism $f_*:\pi_1(S)\to\pi_1(M)$.
It is easy to see that equivalent immersions give rise
to couples $(\tilde{f},\rho)$ in the same $\PSL_2(\CC)$-orbit.

Vice versa, given a couple $(\tilde{f},\rho)$,
there exists $\epsilon>0$ such that the $\rho$-equivariant 
immersion $\tilde{f}:\widetilde{S}\times\{0\}\cong\widetilde{S}\rar\HH^3$ can be extended to an immersion
$\hat{f}:\widetilde{U}=\widetilde{S}\times (-\epsilon,\epsilon)\rar\HH^3$
in such a way that the restriction to each segment $\{\tilde{p}\}\times (-\epsilon,\epsilon)$ is a geodesic of unit speed normal to
$d\tilde{f}_{\tilde{p}}(T_{\tilde{p}}\widetilde{S})$ at $\tilde{f}(\tilde{p})$.
Pulling back the metric of $\HH^3$ via $\hat{f}$, we obtain
a hyperbolic structure on $\widetilde{U}$. By $\rho$-equivariance,
such $\hat{f}$ descends to an immersion of $S$
inside a hyperbolic $3$-manifold $U=S\times(-\epsilon,\epsilon)$.
Clearly, if $(\tilde{f}_1,\rho_1)$ and $(\tilde{f}_2,\rho_2)$
are in the same $\PSL_2(\CC)$-orbit, then they give rise
to the same equivalence class of immersions.
\end{remark}



Given the class $[\tilde{f},\rho]\in \AImm$ of an immersion of $S$ in a germ of hyperbolic manifolds. We denote by $\tilde{a}:T\widetilde{S}\to T\widetilde{S}$ the shape operator of $\tilde{f}$, which is
then self-adjoint with respect to the first fundamental
form of the immersion. It is immediate
to see by $\rho$-equivariance of $\tilde{f}$ that $\tilde{a}$
descends to an operator $a:TS\to TS$.
Thus, to every $[\tilde{f},\rho]$ we can associate a pair $(b,a)$
of bundle morphisms $b,a:TS\to TS$, where $b$ is the operator
defined in Section \ref{ssc:minimizing}.\\

Consider now the locus $\AMImm$ the locus of minimizing immersions inside $\AImm$ and let $\MImm$ be its quotient by $\PSL_2(\CC)$.

\begin{defi}[Immersion datum associated to a minimizing immersion]
For every $[\tilde{f},\rho]$ in $\MImm$ we
define the  $1$-form on $S$ with values on the bundle $T_{\C}S:=\C\otimes_{\R}TS$
$$ \Phi(\tilde{f},\rho) := b-iJba~, $$
where $J$ is the almost-complex structure on $S$ associated to $h$.

  Such $\Phi(\tilde{f},\rho)$ is independent of the chosen representative
  in $[\tilde{f},\rho]$. 
\end{defi}

Notice that $\Phi(\tilde{f},\rho)$ can be uniquely extended to a complex-linear endomorphism of the bundle $T_\C S$. We will often consider such an extension and we denote it by the same symbol.

We will show in Section \ref{sec:euler-lagrange} that minimizing immersions of $(S,h)$ into a (germ of a) hyperbolic manifold can be described in terms of their immersion data.
The equations satisfied by
such minimizing immersion data describe a complex space
$\Data$ defined below.

\begin{defi}[Space of minimizing immersion data]\label{def:D}
Fix a hyperbolic surface $(S,h)$.
The {\it{space of minimizing immersion data}} $\Data$ is the space of smooth $h$-self-adjoint complex-linear operators $\phi:T_\C S\to T_\CC S$
whose real part $\Re(\phi)$ is positive and that
satisfy $d^\nabla \phi=0$ and $\det\phi=1$.
\end{defi}

Here $\nabla$ is the Levi-Civita connection of $h$, which can be extended as a connection on the complex bundle $T_{\C}S$ by $\C$-linearity.
The operator $d^\nabla$ is the exterior derivative on $TS$-valued $1$-forms defined through $\nabla$, so that $d^\nabla\phi$ is a $2$-form on $S$ with values on $T_{\C}S$:
more explicitly, if $v,w$ are two vector fields on $S$
we have
$$ (d^\nabla\phi)(v,w) = \nabla_v (\phi(w)) - \nabla_w(\phi(v)) - \phi([v,w])~. $$

Finally, $h$ can be extended to a complex bilinear form on $T_{\C} S$, still denoted by $h$, so that $\phi$ is $h$-self-adjoint, i.e.~
$h(\phi(v), w)=h(v,\phi(w))$ for all vector fields $v,w$ on $S$.

\subsection{Main results}

%
%
The first main result of this paper consists of a characterization
of the immersion data corresponding to minimizing immersions.

\begin{mainthm}{A}[Immersion data of minimizing immersions]\label{main:immersion-data}
Let $(S,h)$ be a hyperbolic surface.
  \begin{itemize}
  \item[(i)] 
  For every class $[\tilde f,\rho]\in \MImm$ of minimizing immersions, 
  the immersion datum $\Phi([\tilde f,\rho])$ belongs to $\Data$.
  Moreover, each $\phi\in \Data$ is obtained from a unique 
  minimizing class $[\tilde f,\rho]\in \MImm$.
  \item[(ii)]
  The map $\Phi:\MImm=\{[\tilde{f},\rho]\}\rightarrow \Data$ that associates to a minimizing immersion $[\tilde{f},\rho]$ its immersion datum is a homeomorphism.
  \end{itemize}
\end{mainthm}

The correspondence established in Theorem \ref{main:immersion-data} is in fact smooth in the following sense:
\begin{itemize}
\item if a family $(\tilde{f}_t,\rho_t)_{t\in I}$ in $\AMImm$ depends $C^\infty$ on $t$, then the corresponding embedding data $(\phi_t)_{t\in I}$ depends $C^\infty$ on $t$ too;
\item if $(\phi_t)_{t\in I}$ is a $C^\infty$ family of embedding data in $\Data$ and $(\tilde{f}_0,\rho_0)$ corresponds to $\phi_0$, then $(\tilde{f}_0,\rho_0)$ can be  deformed  to a $C^\infty$ family $(\tilde{f}_t,\rho_t)_{t\in I}$ in $\AMImm$ with embedding data $(\phi_t)_{t\in I}$.
\end{itemize}

%

In our second main result we show that the space of immersion data of minimizing maps has a natural structure of complex manifold.\\

Denote by $\cQ$ the space of $J$-holomorphic quadratic differentials on $S$, viewed as a real vector space, and by $\cQ_{\CC}:=\C\otimes_{\R}\cQ$ the complexification of $\cQ$.

We now consider more closely the space of immersion data on $(S,h)$. To do this, we use the decomposition given in Proposition \ref{pr:decomposition}: 
for every smooth $\phi:T_\C S\to T_\C S$ which is self-adjoint and Codazzi, there exists a unique triple $(u,q,q')$, where $u:S\to \C$ is a smooth function and $q,q'\in \cQ$ such that
$$ \phi = (u\Id-\Hess(u))+(b_q+ib_{q'})\,, $$
where
\begin{itemize}
\item 
$\Id$ is the identity operator;
\item 
$\Hess(u)=\nabla(\grad u):TS\to TS$ is the bundle morphism associated (through the Riemannian metric $h$) to the covariant Hessian of $u$;
\item 
$b_q:TS\to TS$ is the bundle morphism associated (through $h$) to the bilinear form $\Re(q)$ on $TS$.
\end{itemize}
In other words, the complex vector space  $\Cod$ of smooth
$d^\nabla$-closed $h$-self-adjoint $1$-forms with values in $T_{\C}S$ 
(endowed with the smooth topology) splits as
\[
   \Cod=C^\infty(S,\C)\oplus\cQ_{\CC}
\]
We will denote by $Q:\Cod\to\cQ_{\CC}$ the projection to $\cQ_{\CC}$ induced by this splitting.

\begin{mainthm}{B}[Manifold structure on the space of minimizing maps]\label{main:manifold-structure}
Let $(S,h)$ be a hyperbolic surface.
The space $\Data$ of immersion data is a complex submanifold of $\Cod$ of complex dimension $6g-6$.
Moreover, the restriction of $Q$ over $\Data$ is a local biholomorphism.
\end{mainthm}
%
%
%

In our third result we show that the monodromy map
that sends a minimizing immersion datum $\phi$ to the
conjugacy class $[\rho_\phi]$ of the monodromy
of the corresponding germ of hyperbolic 3-manifold is a biholomorphism
onto its open image.

\begin{defi}[Space of non-elementary $\SL_2(\CC)$-representations]
The space $\cX$ of non-elementary representation
is the locus in $\Hom(\pi_1(S),\PSL_2(\CC))/\PSL_2(\CC)$ of conjugacy classes of representations without fixed points in $\ol{\HH}^3$.
\end{defi}


\begin{mainthm}{C}[Monodromy map is holomorphic]\label{main:monodromy-holomorphic}
Let $(S,h)$ be a hyperbolic surface.
For every $\phi\in\Data$, the conjugacy class $[\rho_\phi]$ is non-elementary.
Moreover, the map $\Mon:\Data\rightarrow\cX$ that sends $\phi$ to $[\rho_\phi]$ is a biholomorphism onto an open subset of $\cX$ that contains the Fuchsian locus.
\end{mainthm}

In view of Theorem \ref{main:monodromy-holomorphic}, we
define a functional $\F:\cX\rar\RR_{\geq 0}$ on the representation space $\cX$ as
\[
\F([\rho]):=\inf \{F(\tilde{f})\,|\,[\tilde{f},\rho]\in\Imm\}.
\]

In our fourth result we show that,
in view of Theorem \ref{main:monodromy-holomorphic},
for every hyperbolic metric $h$ on $S$
there exists a suitable open subset of $\cX$
that contains the Fuchsian locus and on which
the functional $\F$
is the real part of a holomorphic function.

\begin{mainthm}{D}[Complexification of the functional]\label{main:complexified-F}
Let $(S,h)$ be a hyperbolic surface.
For every $\phi\in \Data$, we have $\F([\rho_\phi])=\Re\int_S \tr(\phi)\da$.
As a consequence, the restriction of $\F$ to the open subset $\Mon(\Data)$ is the real part of the function $\F_{\CC}:\Mon(\Data)\rightarrow\CC$ defined as
$$ \F_{\CC}([\rho]):=\int_S \tr(\Mon^{-1}(\rho))\da\,,$$
which is holomorphic.
\end{mainthm}

The conjectural link between the holomorphic function $\F_\CC$
associated to a hyperbolic metric $h$ and the complex length
associated to a measured lamination $\lambda$ on $S$
is discussed in Section \ref{ssc:complex-length}.

\subsection{Structure of the paper}
In Section \ref{sec:functional}
we define the main objects of investigation, such as
equivariant maps, immersion data and
the $1$-energy functional $F$. Then we discuss
first-order deformations of equivariant maps,
we prove convexity of $F$ along geodesic displacements
and from that we deduce
uniqueness of smooth minimizing immersions.

In Section \ref{sec:euler-lagrange} we compute
the first-order variation of $F$ and we deduce
Euler-Lagrange equations for minimizing immersions,
thus proving Theorem \ref{main:immersion-data}.
Then
we obtain Theorem \ref{main:manifold-structure} through an
implicit function theorem argument.

In Section \ref{sec:monodromy} we resume the
deformation theory of equivariant immersions
developed in Section \ref{ssc:infdef}, and we rephrase
it in terms of the bundle of local Killing vector fields.
Using such rephrasing, we prove Theorem \ref{main:monodromy-holomorphic} and its immediate consequence, namely
Theorem \ref{main:complexified-F}.

In Section \ref{sec:problems} we list some open problems
and perspectives that came up naturally when working at the present
article. 

Finally, we collect in Appendix \ref{sec:app} some facts
on $1$-Schatten norms of matrices and of families of matrices,
that are used in Section \ref{sec:functional}.




\section{A $1$-energy functional}\label{sec:functional}

In this section we introduce certain the space of equivariant
maps and immersions from the universal cover of the surface $S$
to $\HH^3$ and we define the $1$-energy of such maps.
Then we study the deformation theory of such equivariant maps
and we show that the $1$-energy has a convexity property
with respect to geodesic displacements.
We conclude by showing that equivariant immersions (whenever they exist)
are the unique critical point for the $1$-energy functional among maps with the same
monodromy, and in fact
they are a point of absolute minimum with non-elementary reductive monodromy.
We also show that in the Fuchsian case minimizing maps are exactly
minimal Lagrangian maps between hyperbolic surfaces.

\subsection{Setting}
Let $S$ the a compact, connected, oriented surface with $\chi(S)<0$.
Fix a universal cover $\pi:\widetilde{S}\rightarrow S$ and an identification
between $\pi_1(S)$ and the group $\mathrm{Aut}(\pi)$ of the deck transformations of $\pi$.

\begin{notation}
We will use the symbol $\gamma\in \mathrm{Aut}(\pi)\cong\pi_1(S)$ to denote
an automorphism $\gamma:\widetilde{S}\rar\widetilde{S}$ over the covering space $\pi$,
and by $\gamma_*$ the push-forward operator on vector fields
or other tensors on $\widetilde{S}$ induced by the diffeomorphism $\gamma$.
\end{notation}

Fix also a hyperbolic metric $h$ on $S$ and let $\tilde{h}$
be its pull-back on $\widetilde{S}$, so that $(\widetilde{S},\tilde{h})$ is isometric
to the hyperbolic plane $\HH^2$ and
$\pi_1(S)$ acts on
$(\widetilde{S},\tilde{h})$ via hyperbolic isometries.

\subsection{Maps and immersions}

Given a complete hyperbolic $3$-manifold $M$, we can identify its universal cover $\widetilde{M}$
to $\HH^3$ and the group of orientation-preserving isometries of $\widetilde{M}$ to $\PSL_2(\CC)$.
Note that a continuous map $f:S\rightarrow M$ can be lifted to a 
$\pi_1(S)$-equivariant continuous 
map $\tilde{f}:\widetilde{S}\rar\widetilde{M}=\HH^3$, where $\pi_1(S)$
acts on $\widetilde{S}$ by deck-transformations and on $\HH^3$
through a representation $\rho:\pi_1(S)\rightarrow\PSL_2(\CC)$.

The advantage of the following definition is in its flexibility, since
looking at equivariant maps ``allows $M$ to vary''.

\begin{defi}[Equivariant maps]
An {\it equivariant map} from $\widetilde{S}$ to $\HH^3$
is a couple $(\tilde{f},\rho)$, 
where $\rho:\pi_1(S)\rightarrow\PSL_2(\CC)$
is a representation and
$\tilde{f}:\widetilde{S}\rightarrow\HH^3$
is a $\pi_1(S)$-equivariant smooth map.
\end{defi}


If $(\tilde{f}, \rho)$ is an equivariant map, the representation $\rho$ is called the {\it{monodromy}} of $\tilde{f}$. 
Notice that $\rho$ is determined by $\tilde{f}$ provided that the rank of $d\tilde{f}$ is $2$ at least at one point $\tilde{p}\in\widetilde{S}$.

The space $\AMap$ of smooth equivariant maps is a subset
of $\Hom(\pi_1(S),\PSL_2(\CC))\times C^\infty(\widetilde{S},\HH^3)$, and so it inherits a natural topology.
The group $\PSL_2(\CC)$ of orientation-preserving isometries of $\HH^3$
acts on $\AMap$ as
$g\cdot(\tilde{f},\rho):=(g\circ\tilde{f},g\rho g^{-1})$.
Let $\Map:=\PSL_2(\CC)\backslash\AMap$ and
denote by $[\tilde{f},\rho]$ the class of smooth equivariant maps
up to this $\PSL_2(\CC)$-action.\\


We recall from the introduction that $\cX$ denotes 
the space of $\PSL_2(\CC)$-conjugacy classes of non-elementary
representations $\pi_1(S)\rar\PSL_2(\CC)$.
The following fact is well-known.

\begin{lemma}[Smoothness of the representation space]\label{lemma:rep-smooth}
The space $\cX$ is a complex manifold of dimension $-3\chi(S)$.
\end{lemma}


Lemma \ref{lemma:rep-smooth} allows us to define smooth families
of equivariant maps.

\begin{defi}[Paths of equivariant maps]
A {\it{(germ of a) path of equivariant maps}} 
is a couple $(\boldsymbol{\tilde f}, \boldsymbol{\rho})$, where $\boldsymbol{\tilde f}:\widetilde{S}\times (-\epsilon, \epsilon)\to\HH^3$ and 
$\boldsymbol{\rho}:(-\epsilon, \epsilon)\to\cX$ are smooth maps 
for some $\epsilon>0$
such that
for any $t\in(-\epsilon, \epsilon)$ the restriction   $\tilde{f}_t:=\boldsymbol{\tilde{f}}(\bullet, t)$ and  the representation $\rho_t:=\boldsymbol{\rho}(t)$ form a
smooth equivariant map $(\tilde{f}_t,\rho_t)$.
Such a path $(\boldsymbol{\tilde{f}},\boldsymbol{\rho})$ is
a {\it{deformation}} of the equivariant map $(\tilde{f},\rho)$
if $\tilde{f}=\tilde{f}_0$ and $\rho=\rho_0$; moreover, it is an
{\it{isomonodromic deformation}} if $\rho_t=\rho$
for all $t\in(-\epsilon,\epsilon)$.
\end{defi}

In this paper we will be mainly interested in equivariant immersions, that is, equivariant maps $(\tilde{f},\rho)$ such that $d\tilde{f}$ has rank $2$ at every point. As in the introduction,
we will denote by $\AImm$ the space of equivariant smooth immersions,
which is an open subset of $\AMap$ for the smooth topology.
This in particular implies that if $(\boldsymbol{\tilde{f}},\boldsymbol{\rho})$ is a path of equivariant maps and $f_0$ is an immersion, then $f_t$ is an immersion for $t$ sufficiently small.

Clearly $\AImm$ is preserved by the action of $\PSL_2(\CC)$.
Equivariant immersions that differ by post-composition with an isometry of $\HH^3$ are geometrically equivalent. For this reason, we introduce the quotient space $\Imm= \PSL_2(\CC)\backslash\AImm$.

\subsection{Geometry of equivariant immersions}\label{sec:geometry}

Given an equivariant immersion $(\tilde{f}, \rho)$, the pullback 
$\tilde{I}=\tilde{f}^*h_{\HH^3}$ on $\widetilde{S}$
of the hyperbolic metric $h_{\HH^3}$ of $\HH^3$ is a Riemannian metric on $\widetilde{S}$, which is invariant under the action of $\pi_1(S)$. So $\tilde{I}$ is the lift of a Riemannian metric $I$ on $S$. This metric $I$ is called the {\it{first fundamental form}} of the immersion.


On the other hand, associated to $\tilde{f}$ there is a normal field $\tilde{N}$
on $\widetilde{S}$, which is defined by the conditions that at every $\tilde{p}\in\widetilde{S}$
\begin{itemize}
\item 
$\tilde{N}(\tilde{p})\in T_{\tilde{f}(\tilde{p})}\tilde{f}(\widetilde{S})$ is a unitary vector orthogonal to the image of $d\tilde{f}_{\tilde{p}}$.
\item 
if $\left(e_1(\tilde{p}),e_2(\tilde{p})\right)$ is a positively oriented basis
of $T_{\tilde{p}}\widetilde{S}$, then
$\left(\tilde{N}(\tilde{p}),d\tilde{f}_{\tilde{p}}(e_1(\tilde{p})),d\tilde{f}_{\tilde{p}}(e_2(\tilde{p}))\right)$
is a positively oriented basis of $T_{\tilde{f}(\tilde{p})}\HH^3$
\end{itemize}
Formally, $\tilde{N}$ is a section of the bundle $\Theta_{\tilde{f}}:=\tilde{f}^*T\HH^3$,
which comes endowed with
the pull-back $\nabla^{\HH^3}$ of the Levi-Civita connection of $\HH^3$.
Thus $\nabla^{\HH^3}\tilde{N}$ is a $1$-form on $\tilde S$ with values in $\Theta_{\tilde{f}}$.
Standard arguments show that $\nabla^{\HH^3}_{\tilde{w}}\tilde{N}$ is tangent to the immersion for every $\tilde{w}\in T\widetilde{S}$.
Thus an endomorphism $\tilde a: T\widetilde{S}\to T\widetilde{S}$
is defined by requiring that
\[
\nabla^{\HH^3}_{\tilde{w}}\tilde{N}=d\tilde{f}(\tilde a(\tilde{w}))
\]
for every $\tilde{w}\in T\widetilde{S}$. It is a classical fact that $\tilde a$ is $\tilde I$-self-adjoint.
Moreover the equivariance of $\tilde{N}$ with respect to the action of $\pi_1(S)$ implies that $\tilde a$ is $\pi_1(S)$-invariant, so that
it is the lift of an endomorphism $a:TS\to TS$ called the {\it{shape operator}} of the immersion.

The pair $(I, a)$ are called the {\it{immersion datum}} of $\tilde{f}$.
Since $\HH^3$ has curvature $-1$, the couple $(I,a)$ obeys a system of integrability conditions called the Gauss-Codazzi equations:
\begin{equation}\label{eq:gauss-cod}
\begin{cases}
K_I=\det (a)-1\,,\\
d^{\nabla^I}\!a=0\,,\,
\end{cases}
\end{equation}
where $K_I$ is the intrinsic curvature of the metric $I$ and $d^{\nabla^I}$ is the exterior differential associated to the Levi-Civita connection of $I$.
Namely, $d^{\nabla^I}\!a$ is the $2$-form with values in $TS$ defined by $(d^{\nabla^I}\!a)(v,w)=(\nabla^I_v a)(w)-(\nabla^I_w a)(v)$.


The following classical result states that the space $\Imm$ of equivariant immersions up to the action of $\PSL_2(\CC)$ is naturally identified to the space of solutions of \eqref{eq:gauss-cod} through the correspondence that sends $(\tilde{f},\rho)$ to its immersion datum.

\begin{prop}[Immersions and immersion data]
Two equivariant immersions correspond to the same immersion datum $(I,a)$ if and only if they differ by post-composition by an element of $\PSL_2(\CC)$.

Moreover, if $(I, a)$ is a solution of the Gauss-Codazzi equations \eqref{eq:gauss-cod}, where $I$ is a Riemannian metric and $a$ is an $I$-self-adjoint endomorphism of $TS$, then $(I, a)$ is the immersion datum of some equivariant immersion.
\end{prop}

The above correspondence between immersions and immersion data
can be promoted to a correpondence between paths of immersions.

\begin{prop}[Paths of immersions and of immersion data]
If $(\boldsymbol{\tilde f}, \boldsymbol{\rho})$ is a smooth path of equivariant immersions, then the corresponding family of immersion data $(I_t, a_t)_{t\in(-\epsilon, \epsilon)}$ is smooth, 
i.e.~for any couple of vector fields $X,Y$ over $S$ the functions $(t,p)\mapsto I_t(X(p), Y(p))$ and $(t,p)\to a_t(X(p), Y(p))$ defined on $(-\epsilon, \epsilon)\times S$ are smooth.

Conversely, if $(I_t, a_t)_{t\in(-\epsilon, \epsilon)}$ is a smooth family of immersion data, there is a smooth path $(\boldsymbol{\tilde f}, \boldsymbol{\rho})$ of equivariant immersions such that
$(I_t, a_t)$ are the embedding data of $f_t$ for every $t\in(-\epsilon, \epsilon)$.
\end{prop}

\begin{remark}
As we have fixed a reference metric $h$ over $S$, any other Riemannian metric $I$ over $S$ can be described by an $h$-self-adjoint positive operator $b$ by requiring
that $I(v,w)=h(bv, bw)$ for every $v,w\in TS$. Notice that $b$ is the square root of the operator obtained by ``rising'' an index of $I$ with respect to the background metric $h$. 
So, instead of considering couples $(I, a)$, we can look at pairs $(b,a)$, where $b$ is $h$-self-adjoint and $a$ is $I$-self-adjoint, where $I(v,w):=h(bv,bw)$.
\end{remark}

Finally, we will associate to any section of $\Theta_{\tilde{f}}$ a self-adjoint 
and a skew-self-adjoint endomorphism of $\Theta_{\tilde{f}}$, using the following simple linear algebra lemma (see \cite[Section 5.2]{cyclic}).

\begin{lemma}[A-S decomposition of a linear immersion]\label{lm:L}
Let $V$ be an Euclidean vector space of dimension $3$ and let 
$W$ be a 2-dimensional linear subspace. Then any linear map $L:W\to V$ can be uniquely decomposed as $L=\Ae^L+\Se^L$,
where $\Ae^L:W\to W$ is self-adjoint and $\Se^L$ can be written as
$\Se^L(\bullet):=v\times\bullet$ for some fixed $v\in V$.
\end{lemma}

Given a section $\tilde{X}$ of $\Theta_{\tilde{f}}$,
at every point $\tilde{p}$ of $\widetilde{S}$
we can view $(\nabla^{\HH^3}\tilde{X})_{\tilde{p}}$ as a linear map from $T_{\tilde{f}(\tilde{p})}\tilde{f}(\widetilde{S})$ to $T_{\tilde{f}(\tilde{p})}\HH^3$, where we have
identified $T_{\tilde{p}}\widetilde{S}$ and $T_{\tilde{f}(\tilde{p})}\tilde{f}(\widetilde{S})$
via $d\tilde{f}_{\tilde{p}}$.

\begin{defi}[Self-adjoint and skew-self-adjoint derivative of a section of 
$\Theta_{\tilde{f}}$]\label{def:AS}
Let $\tilde{f}$ be an immersion and
$\tilde{X}$ be a section of $\Theta_{\tilde{f}}$.
The {\it{self-adjoint derivative}} of $\tilde{X}$
is the endomorphism $\Ae^{\tilde{X}}_{\tilde{f}}:T\widetilde{S}\to T\widetilde{S}$ defined as
$\left(\Ae^{\tilde{X}}_{\tilde{f}}\right)_{\tilde{p}}:=\Ae^{(\nabla^{\HH^3}\tilde{X})_{\tilde{p}}}$
 for all $\tilde{p}\in\widetilde{S}$.
The {\it{skew-self-adjoint derivative}} of $\tilde{X}$
is the linear morphism $\Se^{\tilde{X}}_{\tilde{f}}:T\widetilde{S}\to \Theta_{\tilde{f}}$
defined as $\left(\Se^{\tilde{X}}_{\tilde{f}}\right)_{\tilde{p}}:=
\Se^{(\nabla^{\HH^3}\tilde{X})_{\tilde{p}}}$ for all $\tilde{p}\in\widetilde{S}$.
\end{defi}


We will usually denote by $\tilde{X}'$ the section of $\Theta_{\tilde{f}}$
such that $\Se^{\tilde{X}}_{\tilde{f}}(\bullet)=\tilde{X}'\times\bullet$.

\subsection{The $1$-Schatten energy}

In this subsection we introduce a functional on $F$ on the space 
$\Map$ of smooth equivariant maps from $\widetilde{S}$ to $\HH^3$ 
that will be a central object of our investigation in this paper.
We incidentally mention thant such functional can be defined on a space of maps of lower regularity (for example, Lipschitz maps).\\

Given a smooth equivariant map $(\tilde{f},\rho)$, the $1$-Schatten norm of $f$  is defined as the function on $\tilde S$ given by $\tilde{p}\mapsto \|d\tilde{f}_{\tilde{p}}\|_1$,
where $\|d\tilde{f}_{\tilde{p}}\|_1$ denotes the $1$-Schatten norm
of the linear map $d\tilde{f}_{\tilde{p}}$
as defined in Section \ref{sec:schatten-def}.

Clearly, this norm is unchanged if we replace $\tilde{f}$ by $g\circ\tilde{f}$ with $g\in\PSL_2(\CC)$.
Hence, the function $\|d\tilde{f}\|_1$ on $\widetilde{S}$ descends to a function on $S$,
denoted as $\|df\|_1$.

\begin{defi}[$1$-Schatten norm of an equivariant map]
The {\it{$1$-Schatten norm of the equivariant map $(\tilde{f},\rho)$}} is 
the function $\|df\|_1:S\to\RR_{\geq 0}$ defined in such a way that for every $p\in S$
the value $\|df_p\|_1$ agrees with the $1$-Schatten norm $\|d\tilde{f}_{\tilde{p}}\|_1$
of the linear map $d\tilde{f}_{\tilde{p}}:T_{\tilde{p}}\widetilde{S}\to T_{\tilde{f}(\tilde{p})}\HH^3$
where $\tilde{p}\in\widetilde{S}$ is any lift of $p$.
\end{defi}

\begin{notation}
The symbol $\|df\|_1$ associated to an equivariant map $(\tilde{f},\rho)$ 
aims at helping the reader in remembering that $\|df\|_1$ is a well-defined
function on $S$, and not just on $\widetilde{S}$. In general, though, no map $f$
is involved in its definition. 
However, if $\tilde{f}$
is a lift of a map $f:S\to M$ to a complete hyperbolic $3$-manifold $M$, then
$\|df_p\|_1$ is exactly the $1$-Schatten norm of $df_p:T_p S\to T_{f(p)}M$.
\end{notation}

The following is a direct consequence of Remark \ref{rmk:1-schatten-nonsmooth}.

\begin{lemma}[Regularity of $1$-energy density]
Let $(\boldsymbol{\tilde{f}}, \boldsymbol{\rho})$ be a path of equivariant maps. The function $S\times(-\epsilon, \epsilon)\ni(p, t)\mapsto \|(df_t)_p\|_1\in\RR_{\geq 0}$ is 
Lipschitz; moreover, it is smooth at all points $p\in S$
such that $(d\tilde{f}_t)_{\tilde{p}}$ has rank $2$.
\end{lemma}

\begin{remark}[$1$-energy density and $b$-operator]
If $(\tilde{f},\rho)$ is an equivariant immersion with first fundamental form $I$
and let $b$ be the $h$-self-adjoint operator on $TS$ such that $I(v,w)=h(bv, bw)$.
Then its pull-back $\tilde{b}$ to $T\widetilde{S}$ is
the $h$-self-adjoint component in the polar decomposition of $d\tilde{f}$.
Thus $\|df\|_1=\tr(b)$.
\end{remark}

\begin{defi}[$1$-Schatten energy]
The {\it{$1$-Schatten energy}} of a smooth equivariant map $(\tilde{f},\rho)$ in $\HH^3$
is defined as 
\[
F(\tilde{f}):=\int_S \|df\|_1\, \da.
\] 
\end{defi}

\begin{remark}\label{rmk:F-Lipschitz}
The $1$-Schatten energy $F(\tilde{f})$ can be defined
for equivariant maps $\tilde{f}$ of lower regularity, such as
Lipschitz maps (in this case $\|df\|_1$ is bounded measurable).
\end{remark}

Clearly, $F(g\circ \tilde{f})=F(\tilde{f})$ for every $g\in\PSL_2(\CC)$.
We also note that,
for a path of equivariant maps  $(\boldsymbol{\tilde{f}}, \boldsymbol{\rho})$,  the function $t\mapsto F(\tilde{f}_t)$ is smooth at $t_0$ provided that the map $\tilde{f}_{t_0}$ is an immersion.

The following simple and important property is a consequence of Remark 
\ref{rmk:1-schatten-lipschitz}.

\begin{lemma}[$1$-Schatten energy and Lipschitz maps]\label{lemma:C-lipschitz}
If $\tilde{f}:\widetilde{S}\rar\HH^3$ is a smooth equivariant immersion
and $g:\HH^3\to \HH^3$ is $C$-Lipschitz, then 
the $1$-Schatten energy (in the sense of Remark \ref{rmk:F-Lipschitz})
of the Lipschitz map $g\circ\tilde{f}$ satisfies
$\|d(g\circ\tilde{f})\|_{1}\leq C\cdot \|d\tilde{f}\|_{1}$
at almost every point of $\widetilde{S}$.
Hence, $F(g\circ \tilde{f})\leq C\cdot F(\tilde{f})$.
\end{lemma}

\subsection{Minimizing maps and critical points of $F$}

Fix a representation $\rho:\pi_1(S)\rightarrow \PSL_2(\CC)$ throughout the whole section.

Let us denote by $\AMap_\rho$ the space of smooth equivariant maps of $\widetilde S$ into $\HH^3$ with monodromy $\rho$,
equipped with the $C^\infty$ topology, and by $\AImm_\rho$  the subspace of smooth equivariant immersions with monodromy $\rho$. 

\begin{defi}[Minimizing $\rho$-equivariant maps]
A $\rho$-equivariant map $\tilde{f}\in \AMap_\rho$ is {\it{minimizing}} if it realizes the minimum of the functional $F$ on $\AMap_\rho$.
\end{defi}

On $\AImm_\rho$ the functional $F$ is smooth
in the following sense: if $\boldsymbol{\tilde{f}}:(-\epsilon,\epsilon)\times\widetilde{S}\to\HH^3$ is a smooth path of equivariant immersions with constant monodromy $\rho$, then the function $t\mapsto F(f_t)$ is smooth. The following is then immediate.

\begin{lemma}[Minimizing immersions are critical points of $F$]
If $(\tilde{f},\rho)\in\AImm_\rho$ is a $\rho$-equivariant minimizing immersion, then
it is a critical point of $F$, i.e.~
for any isomonodromic deformation $\boldsymbol{\tilde{f}}:(-\epsilon,\epsilon)\times\widetilde{S}\to\HH^3$ of $\tilde{f}$
we have
\[
    \left.\frac{d}{dt}F(\tilde{f}_t)\right|_{t=0}=0~.
\]
\end{lemma}

In Section \ref{ssc:convexity} we will discuss the convexity of $F$ showing that any critical immersion is in fact minimizing (Corollary \ref{cor:uniqueness-minima}).


\subsection{Minimizing maps with Fuchsian monodromy}

In this section we consider the case where the representation
$\rho$ is {\it{Fuchsian}}, that is, $\rho$ is a discrete and faithful representation of $\pi_1(S)$ into $\PSL_2(\RR)\subset\PSL_2(\CC)$.
We identify $\HH^2$ to the totally geodesic plane of $\HH^3$ stabilised by $\PSL_2(\RR)$.

We being with a more general remark (see \cite[Section II.1.3]{notesnotesThurston} for more details on the nearest point retraction).

\begin{lemma}[Nearest point retraction]
Let $K$ be a closed convex subset of $\HH^3$, invariant under the action of $\pi_1(S)$ induced by 
a representation $\pi_1(S)\rar\PSL_2(\CC)$.
Then the nearest point retraction
$r:\HH^3\to K$  is $\pi_1(S)$-equivariant and $1$-Lipschitz.
Moreover, if the representation takes values in $\PSL_2(\RR)$ and $K=\HH^2$, then
$r$ is smooth.
\end{lemma}

Thus for any $\rho$-equivariant map $\tilde{f}$, the composition $r\circ \tilde{f}$ is $\rho$-equivariant and $F(r\circ \tilde{f})\leq F(\tilde{f})$ by Lemma \ref{lemma:C-lipschitz}.
This implies that, if $(\tilde{f}_n)$ is any minimizing sequence in
$\AMap_\rho$,
then $(r\circ \tilde{f}_n)$ is still a minimizing sequence in the space of
$\rho$-equivariant Lipschitz maps from $\widetilde{S}$ to $K$.

Now let $\rho$ be a Fuchsian representation, so that $K=\HH^2$
is a $\rho$-invariant closed convex set.
If $(\tilde{f}_n)$ is any minimizing sequence in $\AMap_\rho$,
then $(r\circ\tilde{f}_n)$ is still a minimizing sequence in $\AMap_\rho$
consisting of maps that take values in $\HH^2$.
Hence, minimizers of $F$ among $\rho$-equivariant maps with values into $\HH^2$ are indeed minimizers of $F$.

We will prove in Section \ref{ssc:equations} the following characterization of $\rho$-equivariant local diffeomorphisms $\widetilde{S}\to\HH^2$ that are $F$-minimizers.

\begin{lemma}[Fuchsian minimizers are minimal Lagrangian]\label{lm:basic}
Let $\rho$ be a Fuchsian representation.
A $\rho$-equivariant local diffeomorphism $\tilde{f}:\widetilde{S}\to\HH^2$ is minimizing if and only it is minimal Lagrangian.
\end{lemma}

In this setting we can interpret the result proved by Schoen in \cite{schoen:role}  in terms of the existence result of good minimizers for $F$ when $\rho$ is Fuchsian.

\begin{theorem}[Fuchsian minimizers]\label{thm:2d}
Let $\rho:\pi_1(S)\to \PSL_2(\R)$ be a Fuchsian representation.
There exists a unique smooth $\rho$-equivariant map $\tilde{f}$ which is a critical point of $F$.
Such $\tilde{f}$ takes values in $\HH^2$ and it
is the lift of the unique minimal Lagrangian map $f:S\rightarrow S^\star:=\HH^2/\rho$ between hyperbolic surfaces. As a consequence, $\tilde{f}$ is a real-analytic diffeomorphism.
\end{theorem}

Uniqueness is in fact a consequence the convexity of the functional proved in Section
\ref{ssc:convexity}.

\subsection{Infinitesimal deformations}\label{ssc:infdef}

In this section we study first-order deformations of a 
(not necessarily equivariant) smooth map $\tilde{f}:\tilde S\to\HH^3$,
namely smooth paths of maps $\boldsymbol{\tilde{f}}$ such that $\tilde{f}_0=\tilde{f}$.

\subsubsection{Deformations of maps}

We recall that $\tilde{f}^*T\HH^3$ is the vector bundle on $\widetilde{S}$
consisting of pairs
$(\tilde{p},v)$, where $\tilde{p}\in\widetilde{S}$ and $v\in T_{\tilde{f}(\tilde{p})}\HH^3$.
As in Section \ref{sec:geometry}, we also denote such vector bundle by
$\Theta_{\tilde{f}}$.

\begin{defi}[Geodesic displacement]
Let $\tilde{X}$ be a smooth section of $\Theta_{\tilde{f}}$.
The {\it{geodesic displacement}} of $\tilde{f}$ along $\tilde{X}$
is the path of smooth maps
$\boldsymbol{\tilde{f}}_{\tilde{X}}:\mathbb R\times\widetilde{S}\to\HH^3$ defined as $\boldsymbol{\tilde{f}}_{\tilde{X}}(t,\tilde p):=\exp_{\tilde{f}(\tilde{p})}(t\cdot\tilde{X}(\tilde p))$.
\end{defi}

Consider any smooth deformation $\boldsymbol{\tilde{f}}=(\tilde{f}_t)_{t\in(-\e,\e)}$ of $\tilde{f}$.
The time-derivative of $\tilde{f}_t$ at $t=0$ determines a smooth section $\tilde{X}$ of $\Theta_{\tilde{f}}$, called the {\it{variational field}} of $\boldsymbol{\tilde{f}}$. 
Conversely, any smooth section $\tilde{X}$ of $\Theta_{\tilde{f}}$ is the variational field
of a smooth deformation of $\tilde{f}$: for example,
$\tilde{X}$ is the variational field of the geodesic displacement $\boldsymbol{\tilde{f}}_{\tilde{X}}$.

We say that two deformations of $\tilde{f}$ agree to first-order if they have the same
variational fields.
Thus, we identify first-order deformations of $\tilde{f}$ with their variational fields.

\begin{remark}
Even if we will not formalise this aspect, we can roughly say that the space of smooth sections
$\Gamma(\Theta_{\tilde{f}})$ can be identified to the tangent space at $\tilde{f}$ of the space of smooth maps $\widetilde{S}\to\HH^3$.
\end{remark}


Since the exponential map of the hyperbolic space induces a diffeomorphism between 
$\HH^3$ and $T_x\HH^3$ for any $x\in\HH^3$, we get the following property.

\begin{lemma}[Uniqueness of geodesic displacements between given maps]
If $\tilde{f}_0, \tilde{f}_1$ are $\rho$-equivariant smooth maps, then there is a unique first-order 
smooth deformation $\tilde{X}$ of $\tilde{f}_0$ such that
$\boldsymbol{\tilde{f}}_{\tilde{X}}(1, \tilde{p})=\tilde{f}_1(p)$.
\end{lemma}
\begin{proof}
It is immediate to check that the only infinitesimal deformation of $\tilde{f}_0$ with the stated properties is defined by the formula
 $\tilde{X}(\tilde{p}):=(\exp_{\tilde{f}_0(\tilde{p})})^{-1}(\tilde{f}_1(\tilde{p}))$.
\end{proof}

\subsubsection{Deformation of representations}

Let  $\boldsymbol{\rho}$ is a smooth deformation of the representation
$\rho\in\cX$, namely a smooth path $\boldsymbol{\rho}:(-\epsilon,\epsilon)\to\cX$
such that $\rho_0=\rho$.

For every $\gamma\in\pi_1(S)$, let $\dotr_\gamma\in\mathfrak{sl}_2(\CC)$ be defined as
\[
\dotr_\gamma:= \frac{d}{dt}\rho_t(\gamma)\rho(\gamma)^{-1}\Big|_{t=0}
\]
We recall that $\mathfrak{sl}(2,\CC)$ can be identified to the Lie algebra of Killing vector fields on $\HH^3$.
Under this identification,
$\dotr_\gamma$ can be regarded as the Killing vector field over $\HH^3$
whose value at $x\in\HH^3$ is the velocity of the curve 
$t\mapsto\rho_t(\gamma)\rho(\gamma)^{-1}(x)$ at time $t=0$.

To the given deformation $\boldsymbol{\rho}$ of $\rho$
we can associate the $\mathfrak{sl}(2,\CC)$-valued function
\[
    \dotr:
    \xymatrix@R=0in{
    \pi_1(S)\ar[r] &\mathfrak{sl}(2,\CC) \\
    \gamma\ar@{|->}[r] & \dotr_\gamma
    }
    \]

It is easy to check that, since $\rho_t$ is a representation for all $t$,
the function $\dotr$ satisfies 
the condition
\begin{equation}\label{eq:cocycle}
\dotr_{\gamma\eta}=\dotr_\gamma+ \mathrm{Ad}_{\rho(\gamma)}\dotr_\eta.
\end{equation}

\begin{defi}[$\rho$-twisted $1$-cocycles]
A {\it{$\mathfrak{sl}_2(\CC)$-valued $\rho$-twisted $1$-cocycle}}
is a function $\dotr:\pi_1(S)\rar\mathfrak{sl}_2(\CC)$ that satisfies \eqref{eq:cocycle}.
The vector space of such functions is denoted by $\coc^1_\rho$.
\end{defi}

We say that two smooth deformations of $\rho$
agree to first order if their associated $1$-cocycles agree.
The following is rather classical and we state it without proof.

\begin{lemma}[Deformations of representations as cocycles]\label{lemma:def-rho}
The map
\[
\left\{\text{first-order deformations of $\rho$}\right\}\lra
\coc^1_\rho
\]
that sends a first-order deformation of $\rho$ to its associated $1$-cocycle
is a bijection.
\end{lemma}

Because of the above lemma,
we will identify the first-order deformation $\boldsymbol{\rho}$ of $\rho$
with its associated cocycle function $\dotr$.

The following fact, which we state without proof, will be useful later.

\begin{lemma}[Existence of an equivariant deformation]\label{lemma:exist-def}
Let $(\tilde{f},\rho)$ be an equivariant map and let
$\boldsymbol{\rho}$ be a deformation of $\rho$.
Then there exists a deformation $\boldsymbol{\tilde{f}}:(-\epsilon,\epsilon)\times\widetilde{S}\rar\HH^3$ of $\tilde{f}$ such that $\tilde{f}_t$ is $\rho_t$-equivariant
for all $t\in(-\epsilon,\epsilon)$.
\end{lemma}

\subsubsection{Deformation of equivariant maps}

If $\tilde{f}$ is an equivariant map with monodromy $\rho$, there is a natural action of $\pi_1(S)$ on $\Theta_{\tilde{f}}$, plays an important role to detect the variational fields of deformations through equivariant maps.

Namely, if $\tilde{X}\in\Gamma(\Theta_{\tilde{f}})$ and $\gamma\in\pi_1(S)$, we set
\[
  \gamma_*\tilde{X}(\tilde{p}):=d(\rho(\gamma))_{\tilde{f}(\gamma^{-1}(\tilde{p}))} \tilde{X}(\gamma^{-1}(\tilde{p}))
\]
for every $\tilde{p}\in\widetilde{S}$. 
We denote by $\Theta_f$ the vector bundle on $S$ obtained as the quotient
of $\Theta_{\tilde{f}}$ by the action of $\pi_1(S)$.
Quite similarly, given a deformation $(\boldsymbol{\tilde{f}},\boldsymbol{\rho})_{t\in(-\epsilon,\epsilon)}$ of $(\tilde{f},\rho)$, we denote by $\boldsymbol{\Theta_{f}}$ the vector bundle
on $(-\epsilon,\epsilon)\times S$ defined as the quotient of $\boldsymbol{\Theta_{\tilde{f}}}$
by the action of $\pi_1(S)$.

\begin{notation}
In general, the symbol $\Theta_f$ is defined without using an $f$. However, if
$\tilde{f}$ is the lift of a map
$f:S\rar M$ to a hyperbolic $3$-manifold $M$, then
$\Theta_f$ identifies to the pull-back of $TM$ via $f$.
The same considerations hold for a vector bundle of type $\boldsymbol{\Theta_{f}}$.
\end{notation}

The vector space $\Gamma(\Theta_f)$ can be identified
to the space $\Gamma(\Theta_{\tilde{f}})^\rho$
of $\rho$-invariant elements in $\Gamma(\Theta_{\tilde{f}})$.
Thus sections $X$ of $\Gamma(\Theta_f)$ correspond
to $\rho$-invariant sections $\tilde{X}$ of $\Theta_{\tilde{f}}$

\begin{lemma}[Isomonodromic deformations]
Let $(\tilde{f},\rho)$ be a smooth equivariant map.
Then the map 
\[
\left\{
\begin{array}{c}
\text{first-order deformations}\\
\text{of $(\tilde{f},\rho)$ inside $\AMap_\rho$}
\end{array}
\right\}
\lra
\Gamma(\Theta_{\tilde{f}})^\rho
\]
that sends
a first-order deformations of $(\tilde{f},\rho)$ inside $\AMap_\rho$
to its corresponding ($\rho$-invariant) variational field $\tilde{X}$ is a bijection.
\end{lemma}
\begin{proof}
Consider a deformation $\boldsymbol{\tilde{f}}$ of $\tilde{f}$ inside $\AMap_\rho$.
Since all $\tilde{f}_t$ are $\rho$-equivariant, so is the corresponding variational field $\tilde{X}$,
which thus belongs to $\Gamma(\Theta_{\tilde{f}})^\rho$.
Vice versa, given $\tilde{X}\in \Gamma(\Theta_{\tilde{f}})^\rho$, we let $\boldsymbol{\tilde{f}}$
be the geodesic displacement associated to $\tilde{X}$. Since $\tilde{X}$ is $\rho$-invariant,
so are the maps $\tilde{f}_t$ for all $t$. Hence, $\boldsymbol{\tilde{f}}$ is a smooth
deformation of $\tilde{f}$ with fixed monodromy $\rho$ and variational field $\tilde{X}$.
\end{proof}

Consider now a smooth deformation $(\boldsymbol{\tilde{f}},\boldsymbol{\rho})$ of $(\tilde{f},\rho)$
in which the monodromy $\rho_t$ need not be the same at all $t$.
Let $\tilde{X}$ be the variational field of $\boldsymbol{\tilde{f}}$ and $\dotr$
the $1$-cocycle attached to $\boldsymbol{\rho}$.
For every $\gamma\in\pi_1(S)$, differentiating the relation
$\rho_t(\gamma)\circ\tilde{f}_t=\tilde{f}_t\circ\gamma$ we obtain
\begin{equation}\label{eq:vect-equiv}
\gamma_* \tilde{X}=\tilde{X}+\dotr_\gamma\big|_{\widetilde{S}}
\end{equation}
where $\dotr_\gamma\Big|_{\widetilde{S}}$ 
is pull-back to $\widetilde{S}$ of the Killing vector field $\dotr_\gamma$ in $\HH^3$,
viewed as a section of $\tilde{f}^*T\HH^3$.

\begin{defi}[$1$-cocycle attached to a deformation of an equivariant map]
A {\it{$1$-cocycle}} associated to $(\tilde{f},\rho)$
is a couple $(\tilde{X},\dotr)$ such that $\dotr\in\coc^1_\rho$ and
Equation \eqref{eq:vect-equiv} is satisfied. Such vector space
of $1$-cocycles is denoted by $\coc^1_{(\tilde{f},\rho)}$.
\end{defi}

\begin{lemma}[First-order deformations of equivariant maps]\label{lm:first-order}
Let $(\tilde{f}, \rho)$ be a smooth equivariant map. 
The application
\[
\left\{\text{first-order equivariant deformations of $(\tilde{f},\rho)$}\right\}
\lra \coc^1_{(\tilde{f},\rho)} 
\]
that sends a first-order deformation to its associated
$1$-cocycle is a bijection.
\end{lemma}

%
%
We have already seen how to associate a cocycle to a deformation: such application
is injective essentially by definition.
Surjectivity of such map is more subtle. The point is  that  it is not true in general that, if $\tilde{X}$ is a variational field on $\tilde{f}$ which satisfies \eqref{eq:vect-equiv} for some $\dotr$, then the geodesic displacement $\boldsymbol{\tilde{f}}_{\tilde{X}}$ is a family of maps
that are equivariant with respect to some deformation $\boldsymbol{\rho}$ of $\rho$. 

\begin{proof}
Let $(\tilde{X},\dotr)$ be an element of $\coc^1_{(\tilde{f},\rho)}$
and let $\boldsymbol{\rho}$ be a deformation of $\rho$ with associated $1$-cocycle $\dotr$.
By Lemma \ref{lemma:exist-def}, there exists
$\boldsymbol{\tilde{\phi}}:(-\epsilon,\epsilon)\times\widetilde{S}\rar\HH^3$
that deforms $\tilde{\phi}$ and such that $\tilde{\phi}_t$ is $\rho_t$-equivariant.
Let $\tilde{X}^\star$ 
be the variational field of $\boldsymbol{\tilde{\phi}}$.
Since both $\tilde{X}$ and $\tilde{X}^\star$ solve Equation \eqref{eq:vect-equiv} with respect to $\rho$ and $\dotr$,
their difference $\tilde{Y}_0:=\tilde{X}-\tilde{X}^\star$ is $\pi_1(S)$-invariant, and so descends to a section $Y_0$ of $\Theta_\phi$.
Since $\boldsymbol{\Theta_{\phi}}$ is smoothly isomorphic to $(-\epsilon,\epsilon)\times \Theta_{\phi}$,
the section $Y_0$ of $\{0\}\times \Theta_{\phi}$ can be extended to a smooth section
$\boldsymbol{Y}$ of $\boldsymbol{\Theta_{\phi}}$.
Consider now the path of maps $\boldsymbol{\tilde{f}}:(-\epsilon,\epsilon)\times\widetilde{S}\rar\HH^3$ defined as
$\tilde{f}_t(\tilde{p}):=\exp_{\tilde{\phi}_t(\tilde{p})}(t\cdot \tilde{Y}_t(\tilde{p}))$
Such family $(\boldsymbol{\tilde{f}},\boldsymbol{\rho})$ is a smooth deformation
of $(\tilde{f},\rho)$ with variational field $\tilde{X}^\star+\tilde{Y}_0=\tilde{X}$, as desired.
\end{proof}

As we wrote above, we consider two equivariant maps
in the same $\PSL_2(\CC)$-orbit as ``geometrically equivalent''.
A typical geometrically trivial deformation $(\boldsymbol{\tilde{f}},\boldsymbol{\rho})$
of $(\tilde{f},\rho)$
can be obtained by setting $\tilde f_t:=g_t \circ \tilde{f}$ and $\rho_t:= \mathrm{Ad}_{g_t}\rho$,
where $\boldsymbol{g}:(-\epsilon,\epsilon)\to\PSL_2(\CC)$ satisfies
$g_0=\Id$ and $\dot{g}_0\in\mathfrak{sl}_2(\CC)$.
In this case, a straightforward computation shows that its associated $(\tilde{X},\dotr)$ satisfy
\[
\tilde{X}=\dot{g}_0\Big|_{\widetilde{S}},
\quad\dotr_\gamma= \dot{g}_0-\mathrm{Ad}_{\rho(\gamma)}\dot{g}_0.
\]

\begin{defi}[$1$-coboundary associated to an equivariant map]
A {\it{$1$-coboundary}} associated to $(\tilde{f},\rho)$
is a couple $(\tilde{X},\dotr)$ such that $\tilde{X}=\dot{g}\Big|_{\widetilde{S}}$ and
$\dotr=\dot{g}-\mathrm{Ad}_{\rho}\dot{g}$ for some $\dot{g}\in\mathfrak{sl}_2(\CC)$.
The vector space of $1$-coboundaries is denoted by $\cob^1_{(\tilde{f},\rho)}$.
\end{defi}

The above discussion can be condensed into the following.

\begin{lemma}[Geometrically trivial first-order deformations]\label{lemma:geom-trivial}
Let $(\tilde{f},\rho)$ be an equivariant map.
The map
\[
\left\{\begin{array}{c}
\text{geometrically trivial}\\
\text{first-order deformations of $(\tilde{f},\rho)$}
\end{array}\right\}\lra\cob^1_{(\tilde{f},\rho)}
\]
is a bijection.
\end{lemma}

We can finally summarize the statement of Lemma \ref{lm:first-order} 
and Lemma \ref{lemma:geom-trivial} in this way.

\begin{cor}[Encoding first-order deformations]\label{cor:encoding-def}
Let $(\tilde{f},\rho)$ be a smooth equivariant map.
\begin{itemize}
\item 
The tangent space $T_{\tilde{f}}\AMap_\rho$ can be identified
to $\Gamma(\Theta_f)$.
\item 
The tangent space $T_{(\tilde{f},\rho)}\AMap$ can be identified
to $\coc^1_{(\tilde{f},\rho)}$.
%
%
\item 
The tangent space at $(\tilde{f},\rho)$ to the $\PSL_2(\CC)$-orbit inside $\AMap$
can be identified to $\cob^1_{(\tilde{f},\rho)}$.
\end{itemize}
Hence, the tangent space at $[\tilde{f},\rho]$ to $\Map$ can be identified
to $\coh^1_{(\tilde{f},\rho)}:=\coc^1_{(\tilde{f},\rho)}/\cob^1_{(\tilde{f},\rho)}$.
\end{cor}

\subsection{Convexity}\label{ssc:convexity}

One of the main properties of the $1$-Schatten energy is that it is convex along geodesic deformations.
The rest of this section will be devoted to proving the following statement.

\begin{prop}[Convexity of the $1$-Schatten energy]\label{pr:convexity}
Let $(\tilde{f},\rho)\in\AMap$ be an equivariant map
and let $\boldsymbol{\tilde{f}}$ be a deformation of $\tilde{f}$ with fixed monodromy $\rho$
and variational field $\tilde{X}\in\Gamma(\Theta_{\tilde{f}})^\rho$.
If $\boldsymbol{\tilde{f}_{\tilde{X}}}$ is the geodesic displacement of $\tilde{f}$ along $\tilde{X}$,
then the $1$-Schatten energy of the $\rho$-equivariant map $\tilde{f}_{\tilde{X},t}$ satisfies
\begin{itemize}
\item[(i)]
the function $t\to F(\tilde{f}_{\tilde{X},t})$ is convex;
\item[(ii)]
if $\tilde{f_0}$ is an immersion,  then $\frac{d^2}{dt^2}F(\tilde{f}_{\tilde{X},t})\Big|_{t=0}>0$.
\end{itemize}
\end{prop}

The proof of Proposition \ref{pr:convexity} is based on some technical lemmas
on a smooth perturbation of the Schatten norm,
contained in Section \ref{ssc:linear-convexity}.
Essentially the same proof shows that the result still holds for
equivariant maps inside a negatively curved complete manifold.

\begin{proof}[Proof of Proposition \ref{pr:convexity}]
Let us fix $\tilde{p}\in \widetilde{S}$.
We will prove that the function $t\mapsto \|d(\tilde{f}_{\tilde{X},t})_{\tilde{p}}\|_1$ is convex; and 
moreover that it is strictly convex at $t=0$, provided that $d\tilde{f}_{\tilde{X},0}$ has rank $2$ at $\tilde{p}$ and $\tilde{X}(\tilde{p})\neq 0$.

Consider the parallel transport $\tilde{\tau}_t: \Theta_{\tilde{f}_{\tilde{X},t}}\rar \Theta_{\tilde{f}_{\tilde{X},0}}$
along geodesics in $\HH^3$.
We can define a smooth family $\boldsymbol{\tilde{s}}$
of sections of $\Hom(T\widetilde{S},\Theta_{\tilde{f}})$ by setting
$\tilde{s}_t:=\tilde{\tau}_t\circ d\tilde{f}_{X,t}:T\widetilde{S}\rar \Theta_{\tilde{f}}$.
Clearly we have that $\|d\tilde{f}_{X,t}\|_1=\|\tilde{s}_t\|_1$.
So we need to prove that the function $t\mapsto \|\tilde{s}_t\|_1$ is convex.\\

{\bf{Claim.}}~{\it{
The family $\boldsymbol{\tilde{s}}:\RR\rar \Hom(T\widetilde{S},\Theta_{\tilde{f}})$
is a solution of the following Cauchy problem
\[
\begin{cases}
\ddot{\tilde{s}} = \Xi\circ\tilde{s}\\
\tilde{s}_0=d\tilde{f}_{\tilde{X},0}\\
\dot{\tilde{s}}_0=\nabla^{\HH^3} \tilde{X}
\end{cases}
\]
where $\nabla^{\HH^3}$ is the connection on $\HH^3$ and $\Xi: 
\Theta_{\tilde{f}_{X,0}}\rar\Theta_{\tilde{f}_{X,0}}$ is the
self-adjoint operator defined by
\[
    \Xi(\bullet):=- R(\bullet, \tilde{X})\tilde{X}
\]
where $R$ is the Riemann curvature tensor of $\HH^3$.
}}\\

Assuming the above claim,
the operator $\Xi$ is nonnegative because
the curvature of $\HH^3$ is, and Proposition \ref{lm:convexity} shows that the function $t\mapsto \|\tilde{s}_t\|_1$ is convex, thus proving (i).\\

In order to verify the above claim, 
fix a point $\tilde{p}\in\widetilde{S}$ and let 
$\tilde{\alpha}$ be the geodesic in $\HH^3$
defined by $\tilde{\alpha}(t):=d\tilde{f}_t(\tilde{p})=
\exp_{\tilde{f}_{\tilde{X},0}(\tilde{p})}t\tilde{X}(\tilde{p})$.

For every fixed $\tilde{v}\in T_{\tilde{p}}\widetilde{S}$.
define a vector field $\tilde{J}$ along $\alpha$ as
$\tilde{J}(t):=d\tilde{f}_{\tilde{X},t}(\tilde{v})$.
Then $\tilde{J}$ is a Jacobi field with
initial conditions $\tilde{J}(0)=d\tilde{f}_{\tilde{X},0}(\tilde{v})$ and $\dot{\tilde{J}}(0)=(\nabla^{\HH^3}_{\tilde{v}}\tilde{X})_{\tilde{p}}$.

Now, let $\{e_i\}$ be a parallel orthonormal frame along $\tilde{\alpha}$. Putting $\tilde{J}(t)=\sum_i c^i(t) e_i(t)$ we have that
\[
   \ddot c^i(t)= -\langle R(\tilde{J}(t),\dot{\tilde{\alpha}}(t))\dot{\tilde{\alpha}}(t), e_i(t)\rangle~.
\]
Now considering that $\tilde{s}_t(\tilde{v})=\sum_i c^i e_i(0)$,  we deduce that 
\[
   \ddot{\tilde{s}}_t(\tilde{v})=-\tilde{\tau}_t
   \left( R(\tilde{\tau}_t^{-1}(\tilde{v}),\dot{\tilde{\alpha}}(t))\dot{\tilde{\alpha}}(t)\right)=-R(\tilde{v},\tilde{X})\tilde{X}=\Xi(\tilde{s}_t(\tilde{v}))~,
\]
where the second equality holds because $R$ and $\dot{\tilde{\alpha}}$ are parallel along the geodesic displacement and $\dot{\tilde{\alpha}}(0)=\tilde{X}$.
Since $\tilde{p}$ and $\tilde{v}$ were arbitrary, the claim is proven.\\

Finally, in order to prove (ii),
notice that 
negativity of the curvature of $\HH^3$ has the following consequence:
at every point $\tilde{p}\in\widetilde{S}$ where $\tilde{X}(\tilde{p})\neq 0$,
the quadratic form 
$T_{\tilde{p}}\widetilde{S}\ni \tilde{v}\mapsto \langle\Xi(\tilde{v}),\tilde{v}\rangle=-\langle R(\tilde{v},\tilde{X}(\tilde{p}))\tilde{X}(\tilde{p}),\tilde{v}\rangle$ is semi-positive definite.
%
%
%
Moreover, if $d\tilde{f}_{\tilde{X},0}$ has rank $2$ at $\tilde{p}$ and $\tilde{X}(\tilde{p})\neq 0$, then $\Xi\circ \tilde{s}_0\neq 0$. Hence, by Proposition~\ref{lm:convexity}
we deduce that
\[
   \frac{d^2\|d\tilde{f}_t(\tilde{p})\|_1}{dt^2}\Big|_{t=0}>0
\]
and the proof is complete.
\end{proof}

Let us draw the first consequences of the convexity property proven above.

\begin{cor}[$\rho$-equivariant critical immersions as minima]\label{cor:uniqueness-minima}
If $\tilde{f}$ is a smooth critical immersion in $\AMap_\rho$, then 
$\tilde{f}$ the unique critical point and the absolute minimum of $F$
among $\rho$-equivariant maps isotopic to $\tilde{f}$.
\end{cor}
\begin{proof}
Let $\tilde{f}_1$ be
a $\rho$-equivariant map $\tilde{f}_1$
different from $\tilde{f}_0=\tilde{f}$.
There is an invariant vector field $\tilde{X}\neq 0$ along $\tilde{f}_0$ such that
the geodesic displacement $\boldsymbol{\tilde{f}_{\tilde{X}}}$  connects $\tilde{f}_0$ to $\tilde{f}_1$.
By Proposition~\ref{pr:convexity}, the function $t\mapsto F(\tilde{f}_{\tilde{X},t})$ is strictly convex.
Such a function though has vanishing first derivative at $t=0$, because $\tilde{f}_0$ is a critical
point for $F$. Hence, its derivative at $t=1$ is strictly positive.
Hence, we deduce that $\tilde{f}_1$ is not a critical point for $F$ and that
$F(\tilde{f}_0)<F(\tilde{f}_1)$.
\end{proof}

Another consequence of the convexity of the $1$-Schatten energy is
the following.



\begin{lemma}[Monodromy of critical immersions]\label{lemma:minimizing-group}
Let $\rho:\pi_1(S)\rar\PSL_2(\CC)$ be a representation.
\begin{itemize}
\item[(i)]
Suppose that there exists a smooth $\rho$-equivariant map $\tilde{f}$
which is critical in $\AMap_\rho$. Then $\rho$ is 
reductive, i.e.~the closure of the image of $\rho$ is a reductive subgroup of $\PSL_2(\CC)$.
\item[(ii)]
Suppose that there exists a smooth $\rho$-equivariant map $\tilde{f}$
which is critical in $\AMap_\rho$ and such that $d\tilde{f}$ 
has rank $2$ at some point. Then $\rho$ is non-elementary.
\end{itemize}
\end{lemma}

\begin{proof}
We analyze case by case the possible types of monodromy representations.

If $\rho$ fixes a point $x$ in $\HH^3$, then it is elementary
and reductive and the unique critical point is the constant
map with value $x$. Hence, both (i) and (ii) hold.

If $\rho$ fixes exactly one point $x_\infty\in \partial\HH^3$, then 
it is elementary but not reductive and a $\rho$-equivariant map $\tilde{f}$ cannot have
image contained in a geodesic that limits to $x_\infty$.
As a consequence, the convexity
properties of $F$ imply that $F$ is 
strictly descreasing along the geodesic displacement from $\tilde{f}$
towards $x_\infty$. Thus, there is no $\rho$-equivariant $\tilde{f}$ which is critical for $F$,
and so (i) and (ii) hold.

If $\rho$ fixes exactly two points in $\partial\HH^3$, then it fixes a geodesic $L$ in $\HH^3$.
Hence, $\rho$ is elementary and reductive and
the closest-point retraction $\Pi$ from $\HH^3$ to $L$ is $1$-Lipschitz,
smooth and $\rho$-equivariant. As a consequence, minimizers must take values in $L$
and so both (i) and (ii) hold.

If $\rho$ does not fix any point in $\HH^3$ or in $\partial\HH^3$,
then $\rho$ is not elementary. In this case, the image of $\rho$ cannot have
a finite-index subgroup that fixes a point. Hence, $\rho$ is reductive.
\end{proof}


\section{Minimizing immersions}\label{sec:euler-lagrange}

The aim of this section is to study the space of smooth minimizing immersions
of a given hyperbolic surface $(S,h)$ into germs of hyperbolic manifolds.
This is equivalent to studying the space $\MImm$ of equivalence classes $[\tilde{f},\rho]$ of minimizing smooth equivariant immersions.
The first step is to describe an element of $\MImm$ as a pair $(b,a)$ of endomorphisms of $TS$
that satisfy some relatively elementary properties. The second step is to notice that
the package consisting of all such properties can be expressed in a particularly simple way 
in terms of the $\CC$-linear operator $\phi=b-iJba$ on the complexified tangent bundle $T_{\CC}S$. The last step is to describe the local structure of the space $\Data$ of all such operators $\phi$, and show that $\Data$ is a complex manifold of dimension $-3\chi(S)$.

\subsection{Euler-Lagrange equations} \label{ssc:equations}

In this section, we fix an equivariant immersion $(\tilde{f},\rho)\in\AImm$, consisting
of a representation $\rho:\pi_{1}(S)\rightarrow \PSL_2(\CC)$ together with a smooth $\rho$-equivariant immersion $\tilde{f}$ of $\widetilde{S}$ into $\HH^3$. 
We further consider a smooth isomonodromic deformation $(\boldsymbol{\tilde{f}})$
of $\tilde{f}$ in $\AImm_\rho$
and we determine here the first-order variation of $F(f_t)$ at $t=0$.

As in the previous section, denote by $\tilde{X}=\frac{d}{dt}\tilde{f}_t\Big|_{t=0}$ 
the variational field of $\tilde{f}$, which descends to a section $X\in\Gamma(\Theta_f)$.
Since $\tilde{f}$ is an immersion, the following is immediate.


\begin{lemma}[Decomposition of the variational field]
The variational field  $X$ can be decomposed as $X=X^T+\nu N$, where 
$\nu:S\rar\RR$ is a function,
$X^T$ is tangent to $S$ and $N\in\Gamma(\Theta_f)$ 
is the positively-oriented unit normal vector.
Moreover, the almost-complex structure $J^I$ on $S$ induced by $I$
coincides with the operator $N\times\bullet$ on $S$.
\end{lemma}

After the notation is set as above, we can state a formula for the first variation of $F$.

\begin{prop}[First-order variation of $F$ at an equivariant immersion]\label{pr:extr}
The first-order variation of $F$ along the 
isomonodromic family $(\bm{\tilde{f}})$ with
variational field $X$ satisfies
\[
\frac{d}{dt}F(\tilde{f}_t)\Big|_{t=0}=\int_S (\nu\cdot\tr(ba) + I(bW,J^I X^T))\da~,
\]
where  $W:=\det(b)^{-1}b^{-1}*d^\nabla b$ and $*:\Lambda^2(T^*S)\otimes TS\to TS$ is the Hodge $*$-operator associated with $h$.
\end{prop}

\begin{proof}
Clearly,
\[
\frac{dF(\tilde{f}_t)}{dt}\Big|_{t=0}=\int_S\tr(\dot b)\da
\]
so we need to compute $\dot{b}$.
In order to do so, we compute the first-order variation of the metric $I_t$ induced on $S$ by $\tilde{f}_t$ at $t=0$ in two different way.
On one hand,
differentiating the identity $I_t=h(b_t,b_t)$ at $t=0$, we get
\[
   \dot I=I(b^{-1}\dot b\bullet, \bullet)+I(\bullet, b^{-1}\dot b\bullet)~.
\]
On the other hand, the self-adjoint derivative $\Ae^{\tilde{X}}_{\tilde{f}}$ of $\tilde{X}$ (see Defintion \ref{def:AS}) satisfies $\Ae^{\tilde{X}}_{\tilde{f}}=\nabla^{\tilde{I}} \tilde{X}^T+\tilde{\nu} \tilde{a}$, and it
thus descends to an operator on $TS$, which we denote by $\Ae^{X}$.
It follows that $\dot{I}$ can be also written as
\[
  \dot I=I(A^X\bullet, \bullet)+I(\bullet, A^X\bullet)~.
\]
Comparing those two equations, we find that $b^{-1}\dot b$ and $A^X$ have the same self-adjoint component, and therefore
\[
  b^{-1}\dot b=A^X+g J^I=\nabla^I X^T+\nu a+\eta J^I~,
\]
where $\eta:S\rar\RR$ is a function.
Since $bJ^I=Jb$ is traceless, it follows that
\[
  \tr(\dot b)=\tr (bA^X)=\tr(b\nabla^I X^T)+\nu\tr(ba)~.
\]
By Lemma 3.7 of \cite{cyclic2}, given any vector field $w$ on $S$,
we have 
$\nabla^I_\bullet w=b^{-1}\nabla_\bullet(bw)+ I(\bullet, W)J^I w$.
Taking $w=X^T$, we obtain
$$ \tr(b\nabla^IX^T)=\Div_h(bX^T) + I(bW, J^IX^T)~$$
and so
\[
\tr(\dot{b})=\Div_h(bX^T) + I(bW, J^IX^T)+\nu\,\tr(ba)
\]
Since the integral of $\Div_h(bX^T)\da$ over $S$ vanishes, the conclusion follows.
\end{proof}

Summarizing, critical points of $F$ corresponds to pairs $(b,a)$ such that $b$ is
an $h$-self-adjoint positive Codazzi operator and $\tr (ba)=0$. If we take in account also the Gauss-Codazzi equation of the immersions, we obtain the following statement, realizing the first step announced at the beginning of the section.

\begin{cor}[Euler-Lagrange equations for a critical immersion]\label{cr:extr}
A pair of operators $(b,a)$ on $(S,h)$ corresponds to 
a critical equivariant immersion $(\tilde{f},\rho)$
if and only if the following equations are satisfied:
\begin{eqnarray}
d^\nabla b=0\,, && d^\nabla(ba)=0\,, \label{eq:str1}\\
\tr(Jb)=0\,, && \tr(ba)=0\,, \label{eq:str2}\\
  \tr(Jb^2a)=0\,, && \det b-\det(ba)=1\,, \label{eq:str3}\\
  b>0. && \label{eq:str4}
\end{eqnarray}
\end{cor}

\begin{proof}
The first equation of (\ref{eq:str1}) and the second equation of
(\ref{eq:str2}) are simply the extremality conditions  following from Proposition \ref{pr:extr}. 

The first equation of (\ref{eq:str2}) is equivalent to the fact that $b$ is $h$-self-adjoint. Condition \eqref{eq:str4} is clearly necessary.

Since $\tilde{a}$ is the shape operator of an immersion
with first fundamental form $\tilde{I}$, we have that $d^{\nabla^I}a=0$.
On the other hand, since $b$ satisfies the Codazzi equation, we have 
$d^{\nabla^I}(\bullet)=b^{-1}d^\nabla  (b\bullet)$
and so $d^{\nabla^I}a=b^{-1}d^\nabla(ba)$: the second equation of (\ref{eq:str1}) follows.

Imposing that $a$ is $I$-self-adjoint, we have that $\tr (J^Ia)=0$. But $J^I=b^{-1}Jb$ and it follows that $\tr (b^{-1}Jba)=0$. Since $b$ is $h$-self-adjoint, 
$b^{-1}=-(\det b)^{-1} JbJ$, and the first equation of (\ref{eq:str3}) follows.

Finally by the Gauss equation, $K_I=K_{\HH^3}+\det(a)=
-1+\det(a)$. On the other hand, since $b$ is a Codazzi operator and $h$ is hyperbolic, $K_I=K_h/\det(b)=-1/\det b$. The second equation of (\ref{eq:str3}) easily follows by comparing these identities. 
\end{proof}

The proof of Lemma \ref{lm:basic} follows directly from this corollary.

\begin{proof}[Proof of Lemma \ref{lm:basic}]
  If $f:\tilde{S}\to \HH^2$ is a $\rho$-equivariant local diffeomorphism, then $a=0$ so Equation \eqref{eq:str1} reduces to $d^\nabla b=0$, \eqref{eq:str2} is equivalent to $b$ being self-adjoint for $h$, and \eqref{eq:str3} to $\det(b)=1$. Those are precisely the conditions for $f$ to be minimal Lagrangian, see e.g. \cite{L6}.
\end{proof}

\subsection{Complex interpretation of the Euler-Lagrange equations}

Given a pair of linear operators $(b,a)$ on $TS$, it is convenient to introduce the operator $\phi:T_\CC S\rar T_\CC S$ on the complexified tangent bundle  $T_\CC S=\CC\otimes_\RR TS$ defined as
\[
   \phi:=b-iJba~.
\]
Denoting (with some abuse) by $\nabla$ the $\CC$-linear extension to $T_\CC S$ of the Levi Civita connection of $h$, the Euler-Lagrange equations in Corollary \ref{cr:extr} can be rephrased in a more compact way, a more precise version of Theorem \ref{main:immersion-data}.

\begin{cor}[Complex Euler-Lagrange equations for a critical immersion]\label{cr:extcplx}
The operator $\phi=b-iJba$ corresponds to a critical equivariant immersion
if and only if $\phi\in\Cod$, i.e.
\begin{eqnarray}
& d^\nabla\phi=0~, \label{eq:strcplx1}\\
& \text{$\phi$ is $h$-self-adjoint } \label{eq:strcplx2}\\
& \det(\phi)=1~, \label{eq:strcplx3}\\
&  \text{$\Re(\phi)$ is positive definite.} \label{eq:strcplx4}
\end{eqnarray}
\end{cor}

\begin{proof}
The fact that \eqref{eq:strcplx1} and \eqref{eq:strcplx2} are equivalent to \eqref{eq:str1} and \eqref{eq:str2} respectively is straightforward.

Equation \eqref{eq:strcplx3} is equivalent to Equation \eqref{eq:str3}.
In fact, if $B,C$ are complex $2\times 2$ matrices and
$J$ is a complex skew-symmetric $2\times 2$ matrix such that $J^2=-\Id$,
then $\det(B)\Id=-B^T JBJ$, and so
$\det(B+C)=\det(B)+\det(C)-\tr(JB^T JC)$. If moreover $B$ is symmetric, we have
\[
\det(B+C)=\det(B)+\det(C)-\tr(JBJC)\,.
\]
As a consequence,
$$ \det(\phi)=\det(b)-\det(ba) +i\tr(Jb^2a)\,. $$
Finally \eqref{eq:strcplx4} is clearly equivalent to \eqref{eq:str4}.
\end{proof}



\begin{remark}
Equation \eqref{eq:strcplx3} in Corollary \ref{cr:extcplx} can be replaced by
$\tr\left((J\phi)^2\right)=-2$.
\end{remark}

\subsection{Local structure of the space of immersion data}

We now turn to a more precise analysis of the space $\Data$, with the goal of proving Theorem \ref{main:manifold-structure} stating that $\Data$ is a finite-dimensional complex manifold. 

\begin{defi}[Space of Codazzi operators]
The {\it{space of complex Codazzi operators}} $\Cod$
is the space of smooth $h$-self-adjoint bundle morphism $\phi:T_\C S\to T_\C S$ satisfying the Codazzi equation $d^\nabla \phi=0$. 
\end{defi}

Clearly $\Cod$ is infinite-dimensional. It contains a subspace that is in one-to-one correspondence with the space $C^\infty(S,\CC)$ of $C^\infty$ complex-valued function on $S$, as shown in the following lemma.


\begin{lemma}[The operator $\cod$]
For every $u\in\C^\infty(S,\CC)$ define
\[
\cod(u): = u\Id - \Hess(u)\in\End(T_\CC S)
\]
where $\Id\in\End(T_\CC S)$ is the identity map, $\Hess(u)\in\End(T_\CC S)$
is defined as $\Hess(u)=\nabla(\grad u)$ and the gradient is computed with respect to the metric $h$.
Then
\begin{itemize}
\item[(a)]
$\cod(u)$ is $h$-self-adjoint and satisfies the Codazzi equation;
\item[(b)]
the map $\cod:C^\infty(S,\CC)\rar \Cod$ is $\CC$-linear,
continuous and injective onto a closed subspace $\Codtr$.
\end{itemize}
\end{lemma}
\begin{proof}
Concerning claim (a), note that $\cod(u)$ is $h$-self-adjoint by definition, because $\Hess(u)$ is.
Consider now two vector fields $v,w$ on $S$. Then
  \begin{eqnarray*}
    (d^\nabla(u\Id))(v,w) & = & \nabla_v((u\Id)(w)) - \nabla_w((u\Id)(v)) - u\Id([v,w]) \\
                          & = & \nabla_v(uw) -\nabla_w(uv) - u(\nabla_vw-\nabla_wv) \\
    & = & du(v)w - du(w)v~. 
  \end{eqnarray*}
  It follows that
  $$ (d^\nabla(u\Id))(v,w) = \da(v,w)J\grad u ~. $$
  On the other hand, since $h$ is a hyperbolic metric on $S$, 
  \begin{eqnarray*}
    (d^\nabla \Hess(u))(v,w) & = & \nabla_v(\nabla_w\grad u) - \nabla_w(\nabla_v \grad u) - \nabla_{[v,w]}\grad u \\
                             & = & - R(v,w)\grad u \\
    & = & \da(v,w)J\grad u~. 
  \end{eqnarray*}
  As a consequence, $d^\nabla(u\Id-\Hess(u))=0$ and so $\cod(u)\in\Cod$.

Concerning claim (b), continuity is obvious,
since both spaces are endowed with the smooth topology.
As for the injectivity, suppose that $\cod(u)=0$ and so $\tr(\cod(u))=0$.
Since $\tr(\cod(u))=2u-\Delta u$ (where the Laplacian
is computed with respect to the metric $h$),
an easy application of the maximum principle shows that $u=0$.

The closure of the image $\Codtr$ of the map $\cod$ is a consequence
of Proposition \ref{pr:decomposition} below.
\end{proof}

%
%
%
%

We can now provide a canonical decomposition of self-adjoint complex Codazzi tensors on $S$ which will be useful for the proof of Theorem \ref{main:manifold-structure}, see \cite{bonsante-seppi:codazzi}.

\begin{prop}[Canonical decomposition of Codazzi operators]\label{pr:decomposition}
The map
\[
\Hod:\xymatrix@R=0in{
C^\infty(S,\CC)\times \cQ_\CC\ar[r] & \Cod\\
(u,q,q')\ar@{|->}[r] & \cod(u)+b_q+ib_{q'}
}
\]
is an $\RR$-linear isomorphism,
where $\cQ_\CC:=\cQ\otimes_\RR\CC$ is the complexification of the space of $J$-holomorphic quadratic differentials on $S$,
and $b_q,b_{q'}$ are the endomorphism of $TS$ associated to the
quadratic forms $\Re(q),\Re(q')$ via $h$.
%
\end{prop}

\begin{proof}
We recall that endomorphisms of $TS$ which are $h$-self-adjoint, traceless and satisfy Codazzi
are exactly of type $b_q$ for some holomorphic quadratic differential $q$.
Hence, $\Hod$ is well-defined.

Note that $\Hod$ is continuous and $\RR$-linear.
In order to describe its inverse, let $\phi\in\Cod$.
Define $\tau:=\tr(\phi)$ and denote by $\Delta u = \tr(\Hess(u))$ the Laplacian of $u$ (with respect to $h$). The equation
  $$ -\Delta u +2u=\tau $$
  has a unique solution, since $\Delta$ as defined here is a negative operator
  and so $-\Delta+2$ is invertibile.
Clearly $\phi-\cod(u)$ is self-adjoint and Codazzi, and is also traceless since by construction $\tr(\cod(u))=-\Delta u+2u=\tr(\phi)$. The real and imaginary parts of $\phi-\cod(u)$ are therefore self-adjoint, Codazzi and traceless, and so each of them is equal to the endomorphism
associated to the real part of a holomorphic quadratic differential, as claimed.
%
%
It is easy to see that $\Hod^{-1}$ is continuous, since $(-\Delta+2)^{-1}:C^\infty(S,\CC)\rar C^\infty(S,\CC)$ is a continuous operator. 
\end{proof}

As mentioned in the introduction, the previous proposition allows for the definition of a map
$$ Q:\Cod\to \cQ_\C~, $$
simply by composing $\Hod^{-1}$ with the projection onto $\cQ_\CC$.

Recalling that $\Data\subset\Cod$,
we are now fully equipped to state Theorem B, which we recall here.

\begin{mainthm}{B}[Manifold structure on the space of minimizing maps]
Let $(S,h)$ be a hyperbolic surface.
The space $\Data$ of immersion data is a complex submanifold of $\Cod$ of complex dimension $6g-6$.
Moreover, the restriction of $Q$ over $\Data$ is a local biholomorphism.
\end{mainthm}

The proof will use an additional map that extends the $Q$ defined above.
We denote by $\Pi$ the map
\[
\Pi: \xymatrix@R=0in{
\Cod\ar[r] & C^\infty(S,\CC)\oplus \cQC\\
\phi\ar@{|->}[r] & \big(\Pi_1(\phi),Q(\phi)\big), 
}
\]
where $\Pi_1:\Cod\rar C^\infty(S,\CC)$ is defined as
\[
\Pi_1(\phi):=\tr((J\phi)^2)\,. 
\]
Since $\Data=\Pi_1^{-1}(2)$, Theorem B is an immediate consequence of the following proposition.

\begin{prop}[$\Pi$ is a local diffeomorphism]
Fix $\phi\in \Data$. Then $\Pi$ is a local diffeomorphism at $\phi$.
As a consequence, the restriction of $Q$ to $\Data$ is a local diffeomorphism at $\phi$.
\end{prop}

\begin{proof}
It is sufficient to show that, if $\phi\in \Data$, then $d\Pi_\phi: \Cod=\Codtr\oplus\cQC \rightarrow C^\infty(S,\CC)\oplus \cQC$ is invertible.
This is equivalent to showing that
$(d\Pi_1)_\phi|_{\Codtr}$
is invertible.
Since $\Codtr$ is the image of $\cod$, it is enough to prove that
\[
L_\phi:=(d\Pi_1)_\phi\circ \cod:C^\infty(S,\CC)\rar C^\infty(S,\CC)
\]
is invertible, which is in fact the content of Lemma \ref{lemma:L} below.
\end{proof}

\begin{lemma}\label{lemma:L}
  For all $\phi\in \Data$, the operator $L_\phi$ is invertible.
\end{lemma}

Before proving Lemma \ref{lemma:L}, we will need the following computation.

\begin{sublemma}\label{sublemma:div}
Let $\varphi:TS\rar TS$ be an $h$-self-adjoint operator that satisfies the Codazzi equation.
Then
\begin{equation}  \label{eq:vw}
\int_S \dot{u}\cdot \tr((J\varphi J)\Hess(\dot{u}')) \da  = \int_S \det(\varphi)\langle \varphi^{-1}\grad(\dot{u})\,,\grad(\dot{u}')\rangle \da
\end{equation}
for all $\dot{u},\dot{u}'\in C^\infty(S,\RR)$.
\end{sublemma}
\begin{proof}
  The key observation is that, since $\varphi$ is a real Codazzi operator from $TS$ to $TS$, then the operator $J\varphi J$ is divergence-free for $h$, i.e.~$\nabla^*(J\varphi J)=0$. 
  Indeed, if $(e_1, e_2)$ is a local orthonormal frame, then
  \begin{eqnarray*}
    \nabla^*(J\varphi J) & = & J\nabla^*(\varphi J) \\
                         & = & -J(\nabla_{e_1}(\varphi Je_1) + \nabla_{e_2}(\varphi Je_2) - \varphi J(\nabla_{e_1}e_1 + \nabla_{e_2}e_2)) \\
                         & = & -J(\nabla_{e_1}(\varphi e_2) - \nabla_{e_2}(\varphi e_1) - \varphi (\nabla_{e_1}e_2-\nabla_{e_2}e_1) \\
                         & = & -J (d^\nabla\varphi) (e_1, e_2) \\
                         & = & 0~.
  \end{eqnarray*}
  As a consequence,
  \begin{eqnarray*}
    \int_S \dot{u} \cdot\tr((J\varphi J)\Hess(\dot{u}')) \da & = & \int_S \langle \dot{u}(J\varphi J),\,
    \nabla \grad(\dot{u}')\rangle \da \\
                                            & = & \int_S \langle \nabla^*(\dot{u} (J\varphi J)),\,\grad(\dot{u}')\rangle\da \\
                                            & = & \int_S \langle -(J\varphi J)\grad(\dot{u}),\,\grad(\dot{u}')\rangle \da ~,
  \end{eqnarray*}
  and the conclusion follows by observing that $J\varphi J = -\det(\varphi)\varphi^{-1}$.
 \end{proof}

%
We now have all the ingredients to prove Lemma \ref{lemma:L}.

\begin{proof}[Proof of Lemma \ref{lemma:L}]
  We denote by $\langle\cdot,\cdot \rangle$ the $L^2$ scalar product on complex-valued functions on $S$, defined by
  $$ \langle u_1,u_2\rangle = \Re\left(\int_S u_1\ol{u}_2 \,\da\right)~, $$
  and by $H^1(S)$ the Sobolev space of complex-valued $L^2$ functions with $L^2$ derivative on $S$. The operator $L_\phi$ extends to 
  a second-order linear differential operator $L_\phi:H^1(S)\rar H^{-1}(S)$
  as $L_\phi=2\tr(J\phi J\cdot\cod(\bullet))$.

  We will show that $L_\phi$ is continuous and strongly elliptic, by proving that there exists a constant $c>1$ such that for all $\dot u\in C^\infty(S, \C)$,
  \begin{equation}
    \label{eq:cc}
    \frac 1c \| \dot u\|^2_{H^1} \leq \Re \langle -L_\phi(\dot u),\dot u\rangle \leq c\| \dot u\|^2_{H^1}~. 
  \end{equation}
  The fact that $L_\phi$ is invertible from $C^\infty(S,\C)$ to itself then follows from standard elliptic regularity arguments (see, for instance,
  Theorem 3 in Section 6.3.1 of \cite{evans}).

  Let $\dot u=\dot u_{_\Re}+i\dot u_{_\Im}$, where $\dot u_{_\Re}$ and $\dot u_{_\Im}$ are smooth real-valued functions on $S$, and let $\phi=\phi_{_\Re}+i\phi_{_\Im}$, where $\phi_{_\Re}$ and $\phi_{_\Im}$ are real operators from $TS$ to $TS$. Then:
  \begin{eqnarray*}
    \Re\langle L_\phi \dot u,\dot u) & = &
    \Re\int_S \ol{\dot{u}}\cdot\tr\Big(J\phi J\cdot\cod(\dot{u})\Big)\da=\\    
&=&     \Re\left(
    \int_S (\dot u_{_\Re}-i\dot u_{_\Im}) \tr\Big(
    \big[J(\phi_{_\Re}+i\phi_{_\Im})J\big]
    \big[ \big(\dot u_{_\Re}\Id-\Hess(\dot u_{_\Re}) \big)+i\big(\dot u_{_\Im}\Id-\Hess(\dot u_{_\Im}) \big) \big]
    \Big)\da
  \right) =\\
                                  &  = & \int_S \dot u_{_\Re} 
     \tr\Big[J\phi_{_\Re}J\big(\dot u_{_\Re} \Id-\Hess(\dot u_{_\Re})\big) \Big] \da + \int_S \dot u_{_\Im} \tr\Big[J\phi_{_\Re}J\big(\dot u_{_\Im}\Id-\Hess(\dot u_{_\Im})\big)\Big] \da \\
    & & - \int_S \dot u_{_\Re} \tr\Big[J\phi_{_\Im}J\big(\dot u_{_\Im}\Id-\Hess(\dot u_{_\Im})\big)\Big] \da + \int_S \dot u_{_\Im} \tr\Big[J\phi_{_\Im}J\big(\dot u_{_\Re}\Id-\Hess(\dot u_{_\Re})\big)\Big] \da~.
  \end{eqnarray*}
  Since the right-hand side in  \eqref{eq:vw} is clearly symmetric in $\dot{u}$ and $\dot{u}'$, the last two summands cancel out. Using Sublemma \ref{sublemma:div}, we obtain that
\begin{eqnarray*}
\Re\langle L_\phi \dot u,\dot u)  & =  &\int_S 
\Big(\dot u_{_\Re}^2 \tr(J\phi_{_\Re}J)+ \dot u_{_\Im}^2 \tr(J\phi_{_\Re}J) \Big)~\da\\
&& -\int_S\Big( \det(\phi_{_\Re})\langle \phi_{_\Re}^{-1}(\grad\dot u_{_\Re}),\grad\dot u_{_\Re}\rangle + \det(\phi_{_\Re})\langle \phi_{_\Re}^{-1}(\grad\dot u_{_\Im}),\grad\dot u_{_\Im}\rangle\Big)~\da 
\end{eqnarray*}
  so that
  \begin{eqnarray*}
 -\Re \langle L_\phi \dot u,\dot u) &=&   \int_S \tr(\phi_{_\Re})\big(\dot u_{_\Re}^2+ \dot u_{_\Im}^2\big)
  ~\da\\
  && +\int_S
   \det(\phi_{_\Re})\Big(\langle \phi_{_\Re}^{-1}(\grad\dot u_{_\Re}),\grad\dot u_{_\Re}\rangle + \langle \phi_{_\Re}^{-1}(\grad\dot u_{_\Im}),\grad\dot u_{_\Im}\rangle\Big)~\da~.
  \end{eqnarray*}
  Equation \eqref{eq:cc} now follows from the fact that $\phi_{_\Re}=\Re(\phi)$ is positive-definite 
  by \eqref{eq:strcplx4}.
\end{proof}






\section{Holomorphicity of the monodromy map}\label{sec:monodromy}

Thanks to Theorem \ref{main:immersion-data}, the space $\MImm$ of minimizing immersions in germs of hyperbolic manifolds can be identified to the space $\Data$ of immersion data.
The aim of this section is to show that, under this identification, the monodromy map $\Mon:\Data \rightarrow \cX$, that sends the datum $\phi\in\Data$ corresponding to the 
$\PSL_2(\CC)$-class of an
immersion $[\tilde{f}_\phi,\rho_\phi]\in\MImm$ to the conjugacy class
of its monodromy $[\rho_\phi]\in\cX$, is a biholomorphism onto its (open) image and so to prove Theorem \ref{main:monodromy-holomorphic}.

In order to do that, we first provide a description of the tangent space to the space $\Imm$ of equivariant immersions in $\HH^3$ (up to the action of $\PSL_2(\CC)$) and then of the locus $\MImm\subset\Imm$ that is more suited to reveal its complex-linear nature, compared to the one given in
Section \ref{sec:euler-lagrange}. Then we show that $d\Mon$ is $\CC$-linear.



\subsection{The bundle of local Killing vector fields on a hyperbolic manifold}
\label{ssc:bundle}

We collect in this section some well-known facts that will be useful below.

Given a point $x\in\HH^3$, we call {\it{local Killing vector fields}} the germs at $x$
of Killing vector fields on $\HH^3$ for the hyperbolic metric. 
The vector space $\EE_x$ of such germs at $x$ has a natural structure of Lie algebra, isomorphic to $\psl_2(\CC)$.

\begin{defi}[Bundle of local Killing vector fields]
The {\it{bundle of local Killing vector fields}} in $\HH^3$ is the bundle $\EE\rar \HH^3$
whose fiber $\EE_x$ at a point $x\in\HH^3$ is the Lie algebra of local Killing vector fields at $x$.
\end{defi}

Via the identification of $\psl_2(\CC)$ with the space of global Killing vector fields on $\HH^3$,
the bundle $\EE$ has a natural trivialization
\[
\psl_2(\CC)\times\HH^3\lra \EE
\]
that sends a 
couple $(\Hkill,x)$ to the germ of $\Hkill$ at $x$. The natural flat connection on $\psl_2(\CC)\times\HH^3$
then induces a flat connection $\DD$ on $\EE$: thus, global flat sections of $\DD$ identify with global Killing vector fields on $\HH^3$. 

The action of $\PSL_2(\C)$ on $\HH^3$ lifts to the product bundle $\psl_2(\CC)\times\HH^3$ via the adjoint action on $\psl_2(\CC)$.
Similarly, it also naturally lifts to $\EE$: if $g\in \PSL_2(\C)$ and $\Hkill\in \EE_x$ is a local Killing vector field at $x\in\HH^3$, the image of $\Hkill$ in $\EE_{g\cdot x}$ is the local Killing vector field $g_*\Hkill$ at $g\cdot x$. 
The above trivialization of $\EE$ is equivariant with respect to such $\PSL_2(\CC)$-actions.

\subsubsection{Identification between $\EE$ and $T_\CC\HH^3$}



There is a very natural {\it{evaluation map}} $\EE\rar T\HH^3$
that, at $x\in\HH^3$, sends a local Killing vector field $\Hkill(x)\in\EE_x$
to its value at $x$.

Such evaluation map can be enriched so
to include first-order derivatives of the local Killing vector field.
Specifically in dimensione $3$ there is an identification
\[
\EE\lra T_\CC \HH^3
\]
that is defined as follows.
Given $x\in \HH^3$,
it sends a local Killing vector field $\Hkill(x)\in \EE_x$ to the unique complex tangent vector $\HX_\Hkill(x)+i\HY_\Hkill(x)\in T_{\CC,x}\HH^3$
that satisfies
\begin{itemize}
\item 
$\HX_\Hkill(x)$ equal to the value at $x$ of $\Hkill(x)$, considered as a Killing vector field defined in the neighborhood of $x$,
\item 
$\HY_\Hkill(x)$ is defined by the condition that 
$\Se^{\HX_\Hkill}=\HY_\Hkill\times\bullet$ (see Lemma \ref{lm:L}),
\end{itemize}
where we denoted by $\nabla^{\HH^3}$ the Levi-Civita connection of the hyperbolic metric on $\HH^3$, and by the same symbol its complexification on $T_\CC\HH^3\cong (T\HH^3)\otimes_\RR \CC$.
Abusing notations a bit, we will still denote by $\DD$ the flat connection on $T_\C \HH^3$ obtained as the image of the connection $\DD$ on $\EE$ through the identification of $\EE$ with $T_\C \HH^3$.


\begin{lemma}[Naturality of $\EE\cong T_\CC\HH^3$]\label{lemma:E}
The identification
\[
\xymatrix@R=0in{
\EE \ar[r] & T_\CC\HH^3\\
\Hkill \ar@{|->}[r] & \HX_\Hkill+i\HY_\Hkill
}
\]
is $\CC$-linear and it is equivariant with
respect to the natural action of $\PSL_2(\CC)$ on $\EE$ and on $T_\CC\HH^3$.
Moreover, the flat connection $\DD$ on $T_\CC\HH^3$ can be expressed as
\[
\DD_\bullet(\HX+i\HY) = \nabla^{\HH^3}_\bullet(\HX+i\HY)+i(\HX+i\HY)\times \bullet
\]
in terms of $\nabla^{\HH^3}$.
\end{lemma}

\begin{proof}
The $\CC$-linearity is easy to check. The relation between
$\DD$ and $\nabla^{\HH^3}$ is proven in \cite{MM}. 
\end{proof}

\subsubsection{The case of equivariant maps}

Fix a universal cover $\widetilde{S}\rightarrow S$ and an equivariant  map $\tilde{f}:\widetilde{S}\rightarrow \HH^3$ with monodromy $\rho:\pi_1(S)\rar\PSL_2(\CC)$.

The bundle $\EE$ pulls-back via $\tilde{f}$ to 
the bundle $\widetilde{E}$ on $\widetilde{S}$ 
(isomorphic to $\psl_2(\CC)\times\widetilde{S}$)
endowed with a flat connection $\widetilde{D}$ and $\pi_1(S)$-action
via $\rho$. Its quotient $E_\rho:=\widetilde{E}/\pi_1(S)$ is a $\psl_2(\CC)$-bundle on $S$ and we denote by $D$ its induced flat connection.

On the other hand, the bundle $\tilde{f}^*T_\CC\HH^3$ carries
a connection still denoted by $\nabla^{\HH^3}$, which is the pull-back via $\tilde{f}$ of the complexified Levi-Civita connection on $T_\CC\HH^3$.

As on $\HH^3$, there is a natural evaluation map $\widetilde{E}\rar \tilde{f}^*T\HH^3$
that can be upgraded to an identification $\widetilde{E}\rar \tilde{f}^*T_\CC\HH^3$
using Lemma \ref{lemma:E}. We denote by
$\nabla^{\widetilde{E}}$ the connection on $\widetilde{E}$
corresponding to $\nabla^{\HH^3}$ on $\tilde{f}^*T_\CC\HH^3$,
and by $\nabla^E$ the induced one on $E$.

\subsection{The application $\tilde{\sigma}$}\label{sec:sigma}

Let $(\tilde{f},\rho)$ be an equivariant immersion of $\widetilde{S}$ into $\HH^3$.
Define
\[
\tilde{\sigma}:
\xymatrix@R=0in{
\Gamma(\tilde{f}^*T\HH^3)\ar[r] & \Gamma(\tilde{f}^*T_\CC\HH^3)\cong\Gamma(\widetilde{E})\\
\tilde{X} \ar@{|->}[r] & \tilde{\sigma}_{\tilde{X}}=\tilde{X}+i\tilde{X}'
}
\]
where $\tilde{X}'$ is the unique vector field that satisfies
$\Se_{\tilde{f}}^{\tilde{X}}=\tilde{X}'\times\bullet$ (see Definition \ref{def:AS}).
The main properties of $\tilde{\sigma}$ are collected in the following statement.

\begin{lemma}[Properties of $\tilde\sigma$]\label{lm:propsigma}
The map $\tilde \sigma$ is $\RR$-linear and it satisfies the following properties:
\begin{itemize}
\item[(i)]
$\tilde\sigma_{\gamma_*\tilde{X}}(\gamma (\tilde{p}))=\mathrm{Ad}_{\rho(\gamma)}\tilde\sigma_{\tilde{X}}(\tilde{p})$ for all $\tilde{p}\in \widetilde S$;
\item[(ii)]
$\tilde X\in\Gamma(\tilde{f}^*T\HH^3)$ is the evaluation of a global Killing vector field $\tilde{\kill}\in\Gamma(\widetilde{E})$ if and only if $\tilde\sigma_{\tilde X}$ is $\tilde{D}$-parallel (and in this case $\tilde\sigma_{\tilde X}=\tilde{\kill}$);
\item[(iii)]
$\tilde \sigma_{\tilde{X}}\in\Gamma(\tilde{f}^*T_{\CC}\HH^3)$ is $\tilde{D}$-parallel if and only if $\Ae_{\tilde{f}}^{\tilde X}=0$;
\end{itemize}
for every $\tilde{X}\in\Gamma(\tilde{f}^*T\HH^3)$.
\end{lemma}
\begin{proof}
The $\RR$-linearity of $\tilde{\sigma}$ follows directly from the definition.\\

As for (i),  we compute $\tilde{\sigma}_{\gamma_*\tilde{X}}(\gamma(\tilde{p}))=
\gamma_*\tilde{X}(\gamma(\tilde{p}))+
i(\gamma_*\tilde{X})'(\gamma(\tilde{p}))$.
Since $\pi_1(S)$ acts on $\tilde{f}^*T\HH^3$ isometrically, it preserves
the cross-product and $\nabla^{\HH^3}$.
Hence, $\Se_{\tilde{f}}^{\gamma_*\tilde{X}}\circ 
\gamma_*=\gamma_*\circ \Se^{\tilde{X}}_{\tilde{f}}$ and so $(\gamma_*\tilde{X})'=\gamma_*\tilde{X}'$.
We have then that $\tilde{\sigma}_{\gamma_*\tilde{X}}(\gamma(\tilde{p}))=
\gamma_*\left((\tilde{X}+i\tilde{X}')(\tilde{p})\right)$.
The conclusion then follows because the identification $\widetilde{E}\cong \tilde{f}^*T_{\CC}\HH^3$ is equivariant with respect to the action of $\pi_1(S)$.\\

About (ii), $\tilde{\sigma}_{\tilde{X}}$ is $D^{\tilde{E}}$-parallel
if and only if there exists a global Killing field $\tilde{\kill}$ such that
$\tilde{\sigma}_{\tilde{X}}=\tilde{\kill}$. This clearly happens
if and only if $\tilde{X}$ is the evaluation of such global Killing vector field $\tilde{\kill}$.
\\

Concerning (iii), suppose first that $\tilde{\sigma}_{\tilde{X}}$ is $\tilde{D}$-parallel. By (ii), the vector field $\tilde{X}$ is the evaluation of a global Killing vector field
and so $\Ae_{\tilde{f}}^{\tilde{X}}=0$.

Vice versa, suppose that $\Ae_{\tilde{f}}^{\tilde{X}}=0$, and so $\nabla^{\HH^3}\tilde{X}=\Se_{\tilde{f}}^{\tilde{X}}$.
Let $\tilde X',\tilde{X}''$ be the sections of $\tilde{f}^*T\HH^3$ 
defined by $\Se_{\tilde{f}}^{\tilde{X}}=\tilde{X}'\times\bullet$
and $\Se_{\tilde{f}}^{\tilde{X}'}=\tilde{X}''\times\bullet$.
Now
\[
\Re(\tilde{D}\tilde\sigma_{\tilde X})=\Re(\tilde{D}(\tilde X+i\tilde X'))=\nabla^{\HH^3}\tilde X-(\tilde X'\times\bullet)=0.
\]
On the other hand, the imaginary part of $\tilde{D}\tilde{\sigma}_{\tilde X}$ is given by
\[
\Im(\tilde{D}\tilde\sigma_{\tilde X})=
\Im(\tilde{D}(\tilde X+i\tilde X'))=
\nabla^{\HH^3}\tilde X'+(\tilde X\times\bullet)=\Ae_{\tilde{f}}^{\tilde X'}+\left((\tilde{X}''+\tilde{X})\times\bullet\right) .
\]
By \cite[Lemma 5.7]{cyclic}, we have that
\[
J^{\tilde{I}}\Ae_{\tilde{f}}^{\tilde{X}'}+\langle \tilde{X}+\tilde{X}'',\tilde{N}\rangle\Id=0
\]
where $\tilde{N}$ is the positive unit vector normal to the immersion.

Since $\Ae_{\tilde{f}}^{\tilde{X}'}$ is $\tilde{I}$-self-adjoint,
$J^{\tilde{I}}\Ae_{\tilde{f}}^{\tilde{X}'}$ is traceless and so
\begin{equation}\label{eq:van}
\Ae_{\tilde{f}}^{\tilde X'}=0~.
\end{equation}
Moreover, $\langle \tilde X+\tilde X'',\tilde{N}\rangle\Id=0$ so that $\tilde X+\tilde X''$ is tangent to the immersion.
On the other hand, by the curvature properties of $\HH^3$ we have
\begin{equation}\label{eq:tang}
R(e_1, e_2)\tilde{X}=(e_1\times e_2)\times\tilde{X}
\end{equation}
for any local frame $(e_1,e_2)$ on $\widetilde{S}$.
Since $\Ae_{\tilde{f}}^{\tilde{X}'}=0$,
we have $\nabla^{\HH^3}\tilde{X}'=(\tilde{X}''\times\bullet)$
and so
\begin{equation}\label{eq:tang2}
R(e_1, e_2)\tilde X=\nabla^{\HH^3}_{e_1}\tilde{X}'\times e_2-\nabla^{\HH^3}_{e_2}\tilde X'\times e_1=(\tilde X''\times e_1)\times e_2-(\tilde X''\times e_2)\times e_1=-(e_1\times e_2)\times \tilde X''.
\end{equation}
Comparing \eqref{eq:tang} and \eqref{eq:tang2}, we deduce that 
$(e_1\times e_2)\times (\tilde{X}+\tilde{X}'')=0$ and so
the tangential part of $\tilde X+\tilde X''$ vanishes.
Since we have seen above that $\tilde X+\tilde X''$ is tangent to the image of $\tilde{f}$, we conclude that $\tilde X+\tilde X''=0$. This identity and \eqref{eq:van} together prove that $\Im(\tilde{D}\tilde\sigma_{\widetilde X})=0$.
\end{proof}

\subsection{A complex viewpoint on first-order deformations of immersions} \label{ssc:42}

Let $\tilde{f}$ be an immersion of $\widetilde{S}$ into $\HH^3$.
We recall that a deformation $\bm{\tilde{f}}=(\tilde{f}_t)_{t\in(-\epsilon,\epsilon)}$ of $\tilde{f}$
determines a variational field $\tilde{X}\in\Gamma(\tilde{f}^*T\HH^3)$ by
Corollary \ref{cor:encoding-def}. Moreover, the deformation $\bm{\tilde{f}}$
is tangent to a $\PSL_2(\CC)$-orbit if and only if
$\tilde{X}$ is the evalution of a global Killing vector field.

The above considerations can be rephrased in terms of the complex
$C^\bullet_{\tilde{D}}(\widetilde{E})$
of $\widetilde{E}$-valued differential forms on $\widetilde{S}$ with differential induced by $\tilde{D}$.

\begin{lemma}[First-order deformations of immersions and $\widetilde{E}$-valued forms]
The sequence
\[
\xymatrix{
0\ar[r]& Z^0_{\tilde{D}}(\widetilde{E})\ar[r]&\Gamma(\tilde{f}^*T\HH^3)\ar[r]^{\ \tilde\theta} & Z^1_{\tilde{D}}(\widetilde E)
}
\]
induced by the evaluation map $\Gamma(\widetilde{E})\rar \Gamma(\tilde{f}^*T\HH^3)$
and by $\tilde\theta$ defined as $\tilde{\theta}_{\widetilde X}:=\tilde{D}\tilde\sigma_{\tilde X}$ is exact.
Moreover, 
first-order deformations of $\tilde{f}$  ``up to the action of $\PSL_2(\CC)$'' identify to
the image of $\tilde{\theta}$ inside $Z^1_{\widetilde{D}}(\widetilde{E})$.
\end{lemma}

\begin{remark}[Non-surjectivity of $\tilde{\theta}$]
It is not true that 
the map $\tilde\theta$ is surjective, since 
elements of $Z^1_{\tilde{D}}(\widetilde{E})$ of type
$\tilde \sigma_{\tilde X}$ are determined by their real part.
\end{remark}

\subsubsection{The equivariant case}

Suppose now that $(\tilde{f},\rho)$ is an equivariant immersion of $\widetilde{S}$
into $\HH^3$ and that
$(\bm{\tilde{f}},\bm{\rho})$ is a deformation of $(\tilde{f},\rho)$.
By Lemma \ref{lm:first-order} and Lemma \ref{lemma:geom-trivial}
\begin{itemize}
\item
$(\tilde{X},\dotr)\in \coc^1_{(\tilde{f},\rho)}$, i.e.~$\gamma_* \tilde{X}=\tilde{X}+\dotr_\gamma\big|_{\widetilde{S}}$, and
\item
$(\bm{\tilde{f}},\bm{\rho})$ is tangent
to a $\PSL_2(\CC)$-orbit if and only if $(\tilde{X},\dotr)\in\cob^1_{(\tilde{f},\rho)}$,
i.e.~there exists a global Killing vector field $\tilde{\kill}\in\Gamma(\widetilde{E})$ such that $\tilde{X}$ is the evaluation of $\tilde{\kill}$ and
$\dotr_\gamma=\tilde{\kill}-\mathrm{Ad}_{\rho(\gamma)}\tilde{\kill}$
\end{itemize}
for all $\gamma\in\pi_1(S)$.

%
%
%
%
%
Applying Lemma \ref{lm:propsigma} and using that $\gamma_*\tilde{X}=\tilde{X}+\dotr_\gamma\Big|_{\widetilde{S}}$, we have that
\begin{equation}\label{eq:equi-sigma}
   \tilde\sigma_{\tilde{X}}(\gamma (\tilde{p}))+\dotr_\gamma(\tilde{p})= \mathrm{Ad}_{\rho(\gamma)}\tilde\sigma_{\tilde{X}}(\tilde{p})\,.
\end{equation}
Applying $\tilde{D}$ to \eqref{eq:equi-sigma} and remembering that $\dotr_\gamma$
is $\tilde{D}$-parallel, we obtain
\begin{equation}\label{eq:equi-theta}
   \tilde\theta_{\tilde{X}}(\gamma (\tilde{p}))\circ (d\gamma)_{\tilde{p}}=\mathrm{Ad}_{\rho(\gamma)}\circ \tilde\theta_{\tilde{X}}(\tilde{p}). 
\end{equation}
Condition \eqref{eq:equi-theta} is in fact equivalent to the fact that $\tilde\theta$ is the lift of a $D$-closed $E$-valued $1$-form on $S$.
Thus, an infinitesimal deformation
$(\widetilde X,\dotr)$ corresponds to a deformation through a 
family of equivariant maps if and only if $\tilde \theta_{\tilde{X}}$ is the lift
of an $E$-valued $D$-closed $1$-form $\theta_{\tilde{X}}$, whose periods then correspond to the infinitesimal variation of the monodromy.
In particular, if $\tilde{X}$ is a $\rho$-invariant section of $\tilde{f}^*T\HH^3$,
then the couple $(\tilde{X},\dotr=0)$
corresponds to an infinitesimal isomonodromic deformation 
and $\tilde{\theta}_{\tilde{X}}$ is the lift of a $D$-exact $E$-valued $1$-form
on $S$, i.e.~
the section $\tilde\sigma_{\tilde X}$ is the lift of a section of $E$.



Using the complex $C^\bullet_D(E)$ is $E$-valued differential forms on $S$
with differential induced by $D$,
we can condense the above observations in the following lemma.

\begin{lemma}[First-order deformations of equivariant immersions and $E$-valued forms]
First-order deformations of the equivariant immersion
$(\tilde{f},\rho)$ correspond to those elements $\theta\in Z^1_D(E)$
that are induced by $\tilde{\theta}_{\tilde{X}}$ for some $\tilde{X}$.
In particular, elements in $B^1_D(E)$ correspond to
first-order deformations that fix the conjugacy class of the monodromy.
Moreover, first-order deformations of $\rho\in\cX$
correspond to elements of $H^1_D(E)$, whose periods  give the infinitesimal deformation of the monodromy.
\end{lemma}

\begin{remark}
Note that $\dotr$ does not have any role in the definition of $\tilde{\theta}_{\tilde{X}}$.
Indeed, $\tilde{\theta}_{\tilde{X}}$ determines $\tilde{X}$ up to adding a vector field
obtained by evaluation a global Killing field, and $\dotr$ can be recovered from
$\dotr_\gamma\Big|_{\widetilde{S}}=\gamma_*\tilde{X}-\tilde{X}$, since $\tilde{f}$
is assumed to be an immersion.
\end{remark}

If $(\tilde{f}_\phi,\rho_\phi)$ corresponds to the immersion datum $\phi\in\Data$,
then
the conclusions drawn in the above lemma
can be also visually synthetized into
the following commutative diagram
\[
\xymatrix{
& T_{(\tilde{f}_\phi,\rho_\phi)}\AMImm \ar@{^(->}[r] \ar@{->>}[d]
& T_{(\tilde{f}_\phi,\rho_\phi)}\AImm \ar@{->>}[d] \ar[r]^{\tilde{\sigma}\qquad\qquad\quad} \ar[rd]^{\Theta} & 
\{\tilde{\sigma} \in \Gamma(\widetilde{E})\,|\,
\tilde{D}\tilde{\sigma}\ \text{is $\rho_\phi$-invariant}\}
 \ar[d]^{\tilde{D}} \\
T_{\phi}\Data 
\ar@{<->}[r] \ar@/_1pc/[rrd]^{d\Mon_{\phi}} & T_{[\tilde{f}_\phi,\rho_\phi]}\MImm \ar@{^(->}[r]  &
T_{[\tilde{f}_\phi,\rho_\phi]}\Imm \ar[r]\ar[d]
& Z^1_D(E) \ar@{->>}[d]\\ 
&& T_{[\rho_\phi]}\cX \ar[r]^{\cong} & H^1_D(E)
}
\]
in which the horizontal arrow in the top-right corner
sends $\tilde{X}$ to $\tilde{\sigma}_{\tilde{X}}$,
and $\Theta$ sends $\tilde{X}$ to $\theta_{\tilde{X}}$.

We remark that the vector spaces in the above diagram
that are endowed with a complex structure are
$T_\phi\Data\cong T_{[\tilde{f}_\phi,\rho_\phi]}\MImm$
and the ones in the right column.
As a consequence, $T_{(\tilde{f}_\phi,\rho_\phi)}\AMImm$ is a complex vector space too.
We will see below that
the map $T_{\phi}\Data\rightarrow Z^1_D(E)$ 
between complex vector spaces
that sends $\dot{\phi}$ associated to a variational field $\tilde{X}$ to $\theta_{\tilde{X}}$ is {\underline{not}} $\CC$-linear in general.

%
%

%
%
%
%
%
%
%
%

\subsection{Complex-linearity of $d\Mon$}

Fix an immersion $(\tilde{f}_\phi,\rho_\phi)$ with
corresponding datum $\phi\in\Data$
and let $\widetilde{E}\rar\widetilde{S}$ and $E\rar S$
be the associated bundles of local Killing vector fields.

Denote by $\tilde{N}$ the section of $\tilde{f}_\phi^*T\HH^3$
representing the positively-oriented unit vector field normal
to the image of $\tilde{f}_\phi$.
Viewing $\tilde{N}$ as a section of $\widetilde{E}$, it is $\rho_\phi$-invariant
and so it descends to a section $N\in \Gamma(E)$.

Consider a tangent vector
$\dot{\phi}\in T_\phi\Data$.
By Theorem \ref{main:immersion-data},
there is a germ of path $t\mapsto \phi+t\cdot\dot{\phi}+o(t)$ of immersion data
which is realized by a deformation $(\bm{\tilde{f}},\bm{\rho})$ of $(\tilde{f}_\phi,\rho_\phi)$.
Denote by $\tilde{X}$ the variational field associated to $\bm{\tilde{f}}$
and by $\theta_1\in Z^1_D(E)$ the $1$-cocycle $\theta_{\tilde{X}}$ associated to $\tilde{X}$.
Similarly, the path $t\mapsto \phi+t\cdot (i\dot{\phi})+o(t)$
is realized by a family of immersions corresponding to the
$1$-cocycle $\theta_i\in Z^1_D(E)$.

The holomorphicity of $\Mon$ will then be a consequence of the following
result.

\begin{theorem}[Relation between $\theta_1$ and $\theta_i$]\label{thm:holo}
There exists a smooth function $\nu:S\rightarrow\CC$ such that
$i\theta_1-\theta_i+D(\nu N)=0$. As a consequence,
$[\theta_i]=i[\theta_1]\in H^1_D(E)$.
\end{theorem}

\begin{remark}
Since the function $\nu$ can be nonzero,
the map $T_\phi\Data\rightarrow Z^1_D(E)$ is not $\CC$-linear in general.
\end{remark}

Since $\tilde{f}$ is an immersion, $T\widetilde{S}$ is a subbundle of $\tilde{f}^*T\HH^3$.
Having identified $\tilde{f}^*T_\CC\HH^3$ to $\widetilde{E}$,
it makes sense to decompose $\theta_1$ 
into a component $\theta_1^T$ tangent to the surface and a normal component.
More explicitly, separating real and imaginary parts as $\theta_1=\Re(\theta_1)+i\IM(\theta_1)$,
we have $\theta_1^T=\Re(\theta_1)^T+i\Im(\theta_1)^T$. Indeed, 
$\theta_1^T$ is a $T_\CC S$-valued $1$-form on $S$.

The following proposition relates $\dot\phi$ with $\theta_1^T$ and will be the key 
point to prove Theorem \ref{thm:holo}.

\begin{prop}[First-order variation of immersion data and $1$-cocycles]\label{prop:holo}
There exists a smooth function $\eta:S\rightarrow \RR$ such that
\[
\dot\phi=b\theta_1^T\,+\, \eta J\phi.
\]
\end{prop}

\begin{proof}
Recall that $\tilde{\theta}_1=\tilde{D}\tilde{\sigma}_{\tilde{X}}=\tilde{D}(\tilde{X}+i\tilde{X}')$,
where $\tilde{X}'$ is defined as above in Section \ref{sec:sigma}.
Since $\tilde{f}$ is fixed, we denote the $\tilde{I}$-self-adjoint derivative $\Ae_{\tilde{f}}$
and the $\tilde{I}$-skew-self-adjoint derivative $\Se_{\tilde{f}}$ just by $\Ae$ and $\Se$.

%

As in Lemma 5.5 of \cite{cyclic}, we have
\[
\dot{\tilde{I}}=2\tilde{I}(\Ae^{\tilde{X}}\bullet,\bullet).
\]
Since
$$ \dot{\tilde{I}}=\tilde{h}(\dot{\tilde{b}}, \tilde{b})+\tilde{h}(\tilde{b},\dot{\tilde{b}})
=\tilde{I}(\tilde{b}^{-1}\dot{\tilde{b}}\bullet,\bullet)+\tilde{I}(\bullet,\tilde{b}^{-1}\dot{\tilde{b}}\bullet)~, $$
the operators $\Ae^{\tilde X}$ and $b^{-1}\dot b$ have the same $\tilde{I}$-self-adjoint component, so
\begin{equation}\label{eq:dotb}
\dot{\tilde{b}}=\tilde{b} \Ae^{\tilde{X}}+\tilde{\eta} \tilde{J}\tilde{b}
\end{equation}
for some smooth function $\tilde{\eta}:\widetilde{S}\rightarrow \RR$.
On the other hand, by Lemma 5.6 of \cite{cyclic} we have
\[
\dot{\tilde{a}}=J^{\tilde{I}} \Ae^{\tilde{X}'}-\langle\tilde{X}+\tilde{X}'',\tilde{N}\rangle\En-\Ae^{\tilde{X}} \tilde{a}~.
\]
Since $\frac{d}{dt} \left(\tilde{b}\tilde{a}\right)=\tilde{b}\dot{\tilde{a}}+\dot{\tilde{b}} \tilde{a}$, we get
\begin{eqnarray}
  \frac{d}{dt} \left(\tilde{b}\tilde{a}\right) & =& 
  \tilde{b}J^{\tilde{I}}\Ae^{\tilde{X}'}-\langle\tilde{X}+\tilde{X}'',\tilde{N}\rangle \tilde{b}-\tilde{b}\Ae^{\tilde{X}} \tilde{a}+
\tilde{b}\Ae^{\tilde{X}} \tilde{a}+\tilde{\eta}\tilde{J}\tilde{b}\tilde{a} \\
 & = & \tilde{J}\tilde{b}\Ae^{\tilde{X}'}-\langle\tilde{X}+\tilde{X}'',\tilde{N}\rangle \tilde{b}+\tilde{\eta}\tilde{J}\tilde{b}\tilde{a}~. \label{eq:dotba}
\end{eqnarray}

Since $\tilde{\phi}=\tilde{b}-i\tilde{J}\tilde{b}\tilde{a}$, 
using (\ref{eq:dotb}) and (\ref{eq:dotba}), we obtain that
\begin{eqnarray*}
\dot{\tilde{\phi}} & = & \tilde{b}\Ae^{\tilde{X}}+\tilde{u}\tilde{J}\tilde{b}+i(\tilde{b}(\Ae^{\tilde{X}'}+\tilde{\eta}\tilde{a})+\langle\tilde{X}+\tilde{X}'',\tilde{N}\rangle \tilde{J}\tilde{b})\\
& = & \tilde{b}\Ae^{\tilde{X}}+\tilde{\eta}\tilde{J}\tilde{b}+i\tilde{b}(\Ae^{\tilde{X}'}+\tilde{\eta}\tilde{a}-\langle\tilde{X}+\tilde{X}'',\tilde{N}\rangle \tilde{b}^{-1}\tilde{J}\tilde{b})\\
& = & \tilde{b}(\Ae^{\tilde{X}}+i(\Ae^{\tilde{X}'}+\langle \tilde{X}+\tilde{X}'',\tilde{N}\rangle J^{\tilde{I}}))+\tilde{\eta} \tilde{J}\tilde{b}+i\tilde{\eta}\tilde{b}\tilde{a}~.
\end{eqnarray*}

Notice that
\begin{eqnarray*}
  \tilde{\theta}_1 & = & \tilde D(\tilde X+i\tilde X') \\
  & = & (\nabla \tilde X - \tilde X'\times \bullet) + i(\nabla \tilde X' + \tilde X\times \bullet)\\
  & = & (\Ae^{\tilde X} + \Se^{\tilde X} - \tilde X'\times \bullet) + i(\Ae^{\tilde X'}+\Se^{\tilde X'} + \tilde X\times \bullet)\\
  & = & \Ae^{\tilde{X}}+i(\Ae^{\tilde{X}'}+(\tilde{X}+\tilde{X}'')\times\bullet)~,
\end{eqnarray*}
and so
$\tilde{\theta}_1^T=\Ae^{\tilde{X}}+i(\Ae^{\tilde{X}'}+\langle\tilde{X}+\tilde{X}'',\tilde{N}\rangle J^{\tilde{I}})$. It follows that
\[
\dot{\tilde{\phi}}=\tilde{b}\tilde{\theta}_1^T+\tilde{\eta}\tilde{J}(\tilde{b}-i\tilde{J}\tilde{b}\tilde{a})=\tilde{b}\tilde{\theta}_1^T+
\tilde{\eta}\tilde{J}\tilde{\phi}.
\]
Since all the other tensors are invariant, $\tilde{\eta}$ must come from
a function $\eta:S\rightarrow \RR$ and the result follows.
\end{proof}

As a consequence of the above proposition, we obtain a relation between
the tangent components of $\theta_1$ and $\theta_i$.

\begin{cor}[Tangent components of $\theta_1$ and $\theta_i$]\label{cor:holo}
There exists a smooth complex valued function $\nu:S\rightarrow\CC$ such that
$i\theta_1^T-\theta_i^T+\nu\cdot DN=0$.
\end{cor}

\begin{proof}
By Proposition \ref{prop:holo} 
there exist smooth functions $\eta_1,\eta_i:S\rightarrow\RR$ such that
\begin{align*}
\dot\phi=b\theta_1^T+\eta_1 J\phi\\
i\dot\phi=b\theta_i^T+\eta_i J\phi
\end{align*}
so we deduce that
$b(i\theta_1 - \theta_i)^T+i\nu\cdot J\phi=0$, where $\nu:=\eta_1+i\eta_i$.
Recalling that $DN=-a+i J^I$, or equivalently $J\phi=-ib\cdot DN$, we obtain that
\[
 b\cdot(i\theta_1^T-\theta_i^T+\nu\cdot DN)=0\,.
\]
Since $b$ is invertible, the result follows.
\end{proof}

The relation obtained in the above corollary is almost the wished one.
In order to take care of the normal component of $\theta_1$ and $\theta_i$, we will need the following.

\begin{lemma}[Vanishing $1$-cocycles are detected by their tangent component]\label{lm:closed}
Let $\tau\in Z^1_D(E)$ be any smooth $1$-cocycle. Then $\tau=0$ if and only if $\tau^T=0$.
\end{lemma}

\begin{proof}
Clearly, if $\tau=0$, then its tangent component vanishes.

Conversely, suppose that $\tau^T=0$, so that
$\tau=\zeta\otimes N$, where 
$\zeta$ is a (complex-valued) $1$-form on $S$.
We want to prove that $\zeta=0$.

Since $0=D\tau=d\zeta\otimes N+\zeta\wedge DN$, and since $DN$ takes values in $T_\CC S$,
we deduce that $d\zeta=0$ and $\zeta\wedge DN=0$.
 
Fix any point of $S$ and take a $I$-orthonormal basis 
$(e_1, e_2)$ of $TS$ at that point, formed by eigenvectors for $a$. 
Then imposing that $(\zeta\wedge DN)(e_1,e_2)=0$,  and using that 
$DN=-a+i J^{I}$, one gets
\[
\zeta\wedge DN=\left( -(\Re \zeta)\wedge a-(\Im\zeta)\wedge J^I\right)
+i\left( (\Re\zeta)\wedge J^I-(\Im\zeta)\wedge a\right)
\]
and so
\begin{eqnarray*}
0=2\Re(\zeta\wedge DN)(e_1,e_2)=-\left( (\Re\zeta)(e_1)\, a(e_2)+(\Im\zeta)(e_1) \,J^I e_2\right)\ +\ 
\left( (\Re\zeta)(e_2)\, a(e_1)+(\Im\zeta)(e_2)\, J^I e_1\right)~,\\
0=2\Im(\zeta\wedge DN)(e_1,e_2)=\left((\Re\zeta)(e_1)\,J^I e_2-(\Im\zeta)(e_1)\,a(e_2)
\right)\ -\ 
\left((\Re\zeta)(e_2)\,J^I e_1-(\Im\zeta)(e_2)\,a(e_1)\right)~.
\end{eqnarray*}
Putting $a(e_i)=\lambda_i e_i$, by the first equation it follows that
\begin{equation}\label{eq:real}
\begin{array}{l}
\lambda_2(\Re\zeta)(e_1)=(\Im\zeta)(e_2)~,\\
-(\Im\zeta)(e_1)=\lambda_1(\Re\zeta)(e_2)~.
\end{array}
\end{equation}
By the second equation 
\begin{equation}\label{eq:imaginary}
\begin{array}{l}
(\Re\zeta)(e_1)=\lambda_1(\Im\zeta)(e_2)~,\\
-\lambda_2(\Im\zeta)(e_1)=(\Re\zeta)(e_2)~.
\end{array}
\end{equation}

Suppose by contradiction that $\zeta\neq 0$.
Then Equations \eqref{eq:real} and \eqref{eq:imaginary} imply that
$\det(a)=\lambda_1 \lambda_2=1$. 
But for a critical immersion $\det(a)<0$. Such contradiction proves the lemma.
\end{proof}

The $\CC$-linearity of $d\Mon$ is then readily obtained.

\begin{proof}[Proof of Theorem \ref{thm:holo}] 
Let $\nu$ be the function given in Corollary \ref{cor:holo}
and consider the $1$-cocycle
$\tau:=i\theta_1-\theta_i+D(\nu N)\in Z^1_D(E)$.
Since
\[
\tau^T=i\theta_1^T-\theta_i^T+\nu \cdot DN,
\]
Corollary \ref{cor:holo} implies that $\tau^T=0$.
Hence, we conclude that $\tau=0$ by Lemma \ref{lm:closed}.
\end{proof}

We can now give a complete proof of Theorem \ref{main:monodromy-holomorphic}.

\begin{proof}[Proof of Theorem \ref{main:monodromy-holomorphic}]
By Lemma \ref{lemma:minimizing-group}, a minimal immersion corresponding
to $\phi\in\Data$ has non-elementary monodromy.
Moreover, the map $\Mon$ is holomorphic by Theorem \ref{thm:holo}
and it is injective by Corollary \ref{cor:uniqueness-minima}.
Since $\Data$ and $\cX$ are complex manifolds of the same dimension,
$\Mon$ is a biholomorphism onto its image, which is in fact an open
subset of $\cX$.
Note finally that such open subset $\Mon(\Data)$ contains the Fuchsian
locus by Theorem \ref{thm:2d}.
\end{proof}

\subsection{The complexified functional}

As in the introduction,
define now the functional
$\F:\cX\rar\RR_{\geq 0}$ as
\[
\F(\rho):=\inf\left\{F(\tilde{f})\ \big|\ [\rho,\tilde{f}]\in\Imm\right\}.
\]
The uniqueness proven in Corollary \ref{cor:uniqueness-minima}
implies that
the map $\Mon$ that sends the immersion datum $\phi$
to the class $[\rho_\phi]$ of the monodromy representation of
the immersion corresponding to $\phi$ is injective.

Thus, we can identify $\Data$ to $\Mon(\Data)\subset\cX$,
so that
\[
\F(\rho_\phi)=F(\phi)=\int_S \tr(\Re(\phi))\da=
\Re\int_S \tr(\phi)\da
\]
for every $\phi\in\Data$,
by the minimality property of Corollary \ref{cor:uniqueness-minima}.
By the complex nature of $\Data$, such $\F$ can be viewed
as the real part of $\F_{\CC}:\Mon(\Data)\rightarrow\CC$
defined as
\[
\F_{\CC}([\rho]):=\int_S \tr(\Mon^{-1}(\rho))\da\,.
\]
We can now prove the last main result of our paper.

\begin{proof}[Proof of Theorem \ref{main:complexified-F}]
In view of the above discussion, we are only left to show that
$\F_{\CC}$ is a holomorphic function.
Note that the map $\Data\rightarrow\CC$ defined as
$\phi\mapsto\int_S \tr(\phi)\da$ is clearly holomorphic.
Thus the conclusion follows, since $\Mon:\Data\rightarrow\Mon(\Data)$
is a biholomorphism by Theorem \ref{main:monodromy-holomorphic}.
\end{proof}

\section{Questions and applications}\label{sec:problems}


Let $\cT(S)$ be the Teichm\"uller space of hyperbolic metrics on $S$,
$\cML(S)$ be the space of measured laminations on $S$
and let $\cQF(S)$
be the quasi-Fuchsian space of $S$, i.e.~the space of quasi-Fuchsian metrics
on $M=S\times\RR$. The Fuchsian locus in $\cQF(S)$ consists of metrics
for which $S\times\{0\}$ is a totally geodesic hyperbolic surface.
Consider only maps $f:S\rar M$ that are homotopy equivalences.

\subsection{Existence of smooth minimizing maps}\label{q:smooth}

Once a hyperbolic metric $h\in \cT(S)$ on $S$ is fixed,
we have seen (Theorem \ref{main:monodromy-holomorphic}) that there exists a neighborhood $\Omega_h$ of the Fuchsian locus in $\cQF(S)$ consisting of quasi-Fuchsian structures $g$ on $M$ for which there exists a smooth minimizing map $f:(S,h)\to (M,h_M)$. This smooth minimizing map is unique by Theorem \ref{main:manifold-structure}.

However, we do not know how large this neighbourhood $\Omega_h$ is.
Moreover, a priori, $\Omega_h$ might depend on $h$.
 
Does $\Omega_h$ coincide with the whole $\cQF(S)$?
Does $\Omega_h$ at least contain all ``almost Fuchsian'' structures (i.e.~metrics $h_M$
for which $(M,h_M)$ contains an embedded minimal surface with principal curvatures in $(-1,1)$)? 

The following less ambitious statement asks whether the neighbourhood $\Omega_h$
can be chosen to be independent of the hyperbolic metric $h$.

\begin{question}\label{q:uniform}
Is there a neighborhood $\Omega$ of the Fuchsian locus in $\cQF(S)$ such that, for all $h_M\in \Omega$ and all $h\in \cT(S)$, there exists a smooth minimizing map from $(S,h)$ to $(M,h_M)$? 
\end{question}

We believe that given any quasi-Fuchsian structure $h_M$ on $M$, there exists a mapping $f:(S,h)\to (M,h_M)$ in the BV class
which is minimizing in a weak sense (but which might be not smooth) -- we believe that the existence of minimizing BV maps can be obtained by relatively standard methods. 


\subsection{Uniform convexity and uniqueness among non-smooth maps}

Once $h\in\cT(S)$ and $h_M\in\cQF(S)$ are fixed,
one can ask whether a (possibly non-smooth) minimizing map
$f:(S,h)\rar (M,h_M)$ is unique.
Note that such question is
equivalent to the uniqueness of the $\rho$-equivariant minimizing
$\tilde{f}:(\widetilde{S},\tilde{h})\rar \HH^3$, where $\rho:\pi_1(S)\rar\PSL_2(\CC)$
is the monodromy representation associated to $M\cong\HH^3/\rho(\pi_1(S))$.

Since our arguments for the uniqueness of minimizing maps require some regularity,
we believe that uniqueness can be proven among maps of class $C^1$.
One can ask whether uniqueness holds among continuous maps which are minimizing in the weak sense. This question can be related to the convexity of the functional $F$ over the space of maps of lower regularity from $(S,h)$ to $(M,h_M)$, if $(M,h_M)$ is a quasi-Fuchsian (or more generally a complete hyperbolic) 3-dimensional manifold.
It is even less clear whether uniqueness of the $\rho$-equivariant minimizing
map can be proven for maps $\widetilde{S}\rar\HH^3$ in the BV class.

\subsection{Relation between $\F_\CC$ and the complex length}\label{ssc:complex-length}

When considering diffeomorphisms between hyperbolic surfaces, 
the $1$-energy $F$ is closely related to the hyperbolic length of measured laminations. Specifically, let $h^\star\in \cT_S$ be
a hyperbolic metric with monodromy $\rho$, and let $(h_n)_{n\in \N}$ be a sequence of hyperbolic metrics such that $t_n \cdot h_n\to \lambda$ in the sense of convergence of the length spectrum, where $\lambda\in\cML(S)$ and $t_n\to 0$. Call $f_n:(S,h_n)\rar (S,h^\star)$ the unique minimal Lagrangian map homotopic to the identity, which can be in fact viewed as a minimizing
embedding inside the Fuchsian 3-manifold $(M,h_M)$
associated to $h^\star$, and denote by $\F_{h_n}([\rho])$ the $1$-enegy $F(f_n)$.

In \cite{cyclic} we proved that $t_n \cdot \F_{h_n}([\rho])\rar \ell_\lambda([\rho])$,
where $\ell_\lambda([\rho])$ is the length of the lamination $\lambda$ in $(S,h^\star)$, or equivalently in $(M,h_M)$.
Such result can be rephrased
by saying that, if $\rho$ is the (Fuchsian) monodromy representation
of $(S,h^\star)$, then
$\F_\bullet([\rho])$ defines a continuous function
\[
\F_\bullet([\rho]):(\RR_+\times\cT(S))\cup\cML(S)\lra\RR
\]
that restricts to $\ell_\bullet(h^\star)$ on $\cML(S)$.\\

Suppose now that $(M,h_M)$ is a fixed quasi-Fuchsian manifold 
with monodromy $\rho$ and that
$(h_n)$ is a sequence of hyperbolic metrics on $S$ with
$t_n \cdot h_n\rar\lambda$ and $t_n\rar 0$ as above. Assume that for all $n$
there exists a minimizing map $f_n:(S,h_n)\rar (M,h_M)$ with associated immersion
datum $\phi_n$
(which would follow, for example, if the answer to Question \ref{q:uniform} was positive).
Denote by $\F_{\CC,h_n}([\rho])$ the complex number $\int_S \tr(\phi_n)\omega_{h_n}$
associated to the minimizing map $f_n$.

%

\begin{question} \label{q:complex}
Does $t_n \cdot \F_{\CC,h_n}([\rho])\rar \ell_{\CC,\lambda}([\rho])$, where
$\ell_{\CC,\lambda}([\rho])$ is the complex length of the lamination $\lambda$
in $(M,h_M)$? 
\end{question}

More ambitiously, fixed a quasi-Fuchsian manifold $(M,h_M)$ with
monodromy $\rho$,
one could ask whether the complex valued functional $\F_{\CC,\bullet}([\rho])$
can be extended so to define a continuous function
\[
\F_{\CC,\bullet}([\rho]):(\RR_+\times\cT(S))\cup\cML(S)\lra\CC
\]
that restricts to the complex length function $\ell_{\CC,\bullet}(M,h_M)$ on $\cML(S)$.


\subsection{Non-quasi-Fuchsian targets}

This above questions are not necessarily restricted to quasi-Fuchsian manifolds -- given a closed 3-dimensional hyperbolic manifold $(M,h_M)$ and a homotopy class of maps from $(S,h)$ into $(M,h_M)$ that induce an injection $\pi_1(S)\hra\pi_1(M)$, one can ask whether it contains a smooth minimizing immersion,
or whether uniqueness holds among minimizing maps of lower regularity. The arguments used to prove uniqueness of smooth minimizing maps in quasi-Fuchsian manifolds also work in this setting.


\appendix

\section{On the $1$-Schatten norm of linear maps}\label{sec:app}

In this section we recall properties of the $1$-Schatten norm
on the space of linear homomorphisms between vector space of finite dimension
endowed with a positive-definite scalar product.

\subsection{Definition and basic properties}\label{sec:schatten-def}

Let $V$ and $W$ be finitely generated vector spaces, equipped with positive-definite scalar products,
and assume that $\dim V\leq\dim W$.
Any linear map $L:V\to W$ can be factorized as the composition $L=\sigma_L\circ b_L$, where $b_L=\sqrt{L^T\circ L}$ 
is a non-negative $g$-self-adjoint endomorphism of $V$, the map $\sigma_L:V\rar W$ is an isometric linear embedding and $L^T:W\to V$ denotes the adjoint of $L$.

\begin{remark}[Polar decomposition]
While $b_L$ is always well-defined,  $\sigma_L$ is uniquely determined provided that $L$ is injective (or equivalently that $\det b_L\neq 0$). In this case we refer to the decomposition $L=\sigma_L\circ b_L$ as the {\it{polar decomposition}} of $L$.
\end{remark}

\begin{defi}[$1$-Schatten norm of a linear map]
The {\it{$1$-Schatten norm of a linear map}} $L:V\to W$ is 
$\|L\|_1:=\tr(b_L)$, where $b_L:=\sqrt{L^T\circ L}$.
\end{defi}

\begin{remark}[Lipschitz nature of $1$-Schatten norm on a $\Hom$-space]\label{rmk:1-schatten-nonsmooth}
The function $\|\cdot\|_1:\Hom(V,W)\rar\RR$ is 
a genuine norm is on the vector space $\Hom(V, W)$.
Moreover, it is Lipschitz (say, with respect to the natural Riemannian metric on $\Hom(V,W)$) but not $C^1$ at homomorphisms
of non-maximal rank and it is a smooth at homomorphism of maximal rank.
\end{remark}

\begin{remark}[$1$-Schatten norm and Lipschitz linear maps]\label{rmk:1-schatten-lipschitz}
If $A$ is a linear endomorphism of $W$ which is $C$-Lipschitz,
then $\|AL\|_1\leq C\cdot \|L\|_1$.
This easily follows from the fact that
$\|ALv\|_W\leq C\cdot \|Lv\|_W$ for all $v\in V$.
\end{remark}

\subsection{Convexity}\label{ssc:linear-convexity}

In order to study the convexity properties of the $1$-Schatten norm,
we consider a smooth perturbation of it. For brevity, our treatment is limited
to the special case we are interested in. There result we want to prove is the following.

\begin{prop}[Convexity of $1$-Schatten norm along paths with positive acceleration] \label{lm:convexity}
Let $V$ and $W$ be vector spaces of dimension $2$ and $3$ respectively and let 
$A:[0,1]\rar \Hom(W,W)$ be a smooth path of positive self-adjoint operators on $W$.
Consider a smooth path
$T:[0,1]\rightarrow \Hom(V,W)$
of linear operators such that $\ddot T(s)=A(s)\circ T(s)$.
Then the function  $u:[0,1]\rar\RR$ defined as $u(s):= \|T(s)\|_{1}$ is convex. 
Moreover, if the rank of $T(s)$ is $2$ and $A(s)\circ T(s)\neq 0$, then $\ddot u(s)>0$.
\end{prop}

Before proving Proposition \ref{lm:convexity}, let us introduce 
suitable regularized
versions of the $1$-Schatten norm and study their properties.\\

Let $V,W$ be real vector spaces endowed with positive-definite scalar products,
of dimension $2$ and $3$ respectively. For any $\epsilon\geq 0$ and for $L\in\Hom(V,W)$
the {\it{$\epsilon$-regularized $1$-Schatten norm}} of $L$ is
\[
   q_\epsilon(L):=\tr\sqrt{\epsilon^2\Id+L^*L}
\]
where $L^*$ is the adjoint of $L$.
Notice that $q_0$ coincides with the $1$-Schatten norm.

\begin{lemma}[Basic properties of regularized $1$-Schatten norms]\label{lm:sh-pr}
The $\epsilon$-regularized $1$-Schatten norm $q_\epsilon:\Hom(V,W)\rightarrow\RR$ satisfies the following properties.
\begin{itemize}
\item[(a)] For any $L\in\Hom(V,W)$,
\begin{equation}\label{eq:sh-ex}
q_\epsilon(L)=\sqrt{\tr(\epsilon^2\Id+ L^*L)+2\sqrt{\det(\epsilon^2 \Id+ L^*L)}}~.
\end{equation}
As a consequence, $q_\epsilon$ is smooth for $\epsilon>0$.
\item[(b)]
The function $q_\epsilon$ is convex for any $\epsilon\geq 0$.
\item[(c)]
Let $L:V\rar W$ be a linear map,
$A:W\rar W$ be nonnegative self-adjoint and let
$\epsilon\geq 0$. In the case $\epsilon=0$, suppose furthermore that $L$ has rank $2$. 
Then 
\[
   \frac{d\,}{dt}q_\epsilon((1+tA)L)\Big|_{t=0}\geq 0~.
\]
Moreover the strict inequality holds if $A\circ L\neq 0$.
\end{itemize}
\end{lemma}

Before proving the lemma, we mention the following observation.

\begin{sublemma}\label{sublemma:N}
Let $\epsilon\geq 0$.
Consider the planar domain $\Omega=\{\bm{t}=(t_1,t_2)\in\RR^2| t_2\geq t_1\geq 0\}$ and
define $n_\epsilon:\Omega\rar\RR$ by
$n_\epsilon(t_1,t_2):=\sqrt{\epsilon^2+t_1^2}+\sqrt{\epsilon^2+t_2^2}$.
Then, $n_\epsilon$ is convex. Moreover,
given $(t_1,t_2), (t'_1,t'_2)\in\Omega$ such that $t_2\leq 2'_1$ and $t_1+t_2\leq t'_1+t'_2$, we have $n_\epsilon(t_1,t_2)\leq n_\epsilon(t'_1,t'_2)$.
\end{sublemma}
\begin{proof}
Clearly, the function $t\mapsto \sqrt{\epsilon^2+t^2}$ is increasing and convex
and so $n_\epsilon$ is convex too.
It follows that $n_\epsilon(t'_1,t'_2)\geq n_\epsilon(t_1,t_2)+
(\pa_{t_1}n_\epsilon)_{\bm{t}}(t'_1-t_1)+(\pa_{t_2}n_\epsilon)_{\bm{t}}(t'_2-t_2)$.
Now,
the simple and key remark is that  $(\partial_{t_2} n_\epsilon)_{\bm{t}}\geq (\partial_{t_1} n_\epsilon)_{\bm{t}}\geq 0$ for every $\bm{t}\in\Omega$. 
Thus, 
$ (\pa_{t_1}n_\epsilon)_{\bm{t}}(t'_1-t_1)+(\pa_{t_2}n_\epsilon)_{\bm{t}}(t'_2-t_2)
\geq (\pa_{t_1}n_\epsilon)_{\bm{t}}(t'_1-t_1+t'_2-t_2)\geq 0 $ and the conclusion follows.
\end{proof}

The relevance of the above sublemma
relies on the fact that, given $L\in\Hom(V,W)$, we have that $q_\epsilon(L)=n_\epsilon(\lambda_1,\lambda_2)$, where $\lambda_1\leq\lambda_2$ are the singular values of $L$.

\begin{proof}[Proof of Lemma \ref{lm:sh-pr}]
If $\lambda_1,\lambda_2$ are the singular values of $L$, that is, the eigenvalues of $b_L=\sqrt{L^*L}$,
then
\[
   q_\epsilon(L)=\sqrt{\epsilon^2+\lambda_1^2}+\sqrt{\epsilon^2+\lambda_2^2}\,.
\]
So $q_\epsilon(L)^2=2\epsilon^2+\lambda_1^2+\lambda_2^2+2\sqrt{(\epsilon^2+\lambda_1^2)(\epsilon^2+\lambda_2^2)}$.
As the eigenvalues of $\epsilon^2\Id+L^*L$ are $\epsilon^2+\lambda_1^2$ and $\epsilon^2+\lambda_2^2$, the identity \eqref{eq:sh-ex} is immediately verified.
Moreover, $q_\epsilon(L)^2$ can be written as
$q_\epsilon(L)^2=2\epsilon^2+\tr(L^*L)+2\sqrt{\epsilon^4+2\epsilon^2\tr(L^*L)+\det(L^*L)}$,
which shows that $q_\epsilon$ is smooth for $\epsilon>0$. This proves (a).

In order to prove (b), note that $q_\epsilon$ is continuous and so it is enough
to show that $2q_\epsilon(L+L')\leq q_\epsilon(2L)+q_\epsilon(2L')$ for all $L,L'\in\Hom(V,W)$.
Consider then
$L, L'\in\Hom(V,W)$ and denote by $\lambda_1\leq\lambda_2$
the singular values of $L$, by
$\lambda'_1\leq \lambda'_2$
the singular values of $L'$ and by
$\mu_1\leq \mu_2$ the singular values of $L+L'$.
By Theorem 2 of \cite{thompson:convex} there are linear isometries $P,Q:V\rar V$  such that
\[
     b_{L+L'}\leq  P^* b_L P+ Q^* b_{L'} Q
\]
where $\leq$ means that the difference is a nonnegative self-adjoint matrix.

If $\hat{\mu}_1\leq\hat{\mu}_2$ are the eigenvalues of 
$P^* b_L P+Q^* b_{L'}Q$, we deduce that $\mu_i\leq\hat{\mu}_i$
for $i=1,2$. Thus $q_\epsilon(L+L')\leq n_\epsilon(\hat{\mu}_1, \hat{\mu}_2)$. 
Moreover, $\hat{\mu}_1+\hat{\mu}_2=\tr(b_{L+L'})\leq \tr(P^* b_L P+Q^* b_{L'}Q)=\tr(b_L)+\tr(b_{L'})=\lambda_1+\lambda'_1+\lambda_2+\lambda'_2$.
On the other hand, by the classical Weyl's Theorem, $\hat{\mu}_2\leq \lambda_2+\lambda'_2$.
Hence, using Sublemma \ref{sublemma:N}, we get
\[
  2q_\epsilon(L+L')\leq 2n_\epsilon(\hat{\mu}_1,\hat{\mu}_2)\leq 2 n_\epsilon(\lambda_1+\lambda'_1,\lambda_2+\lambda'_2)\leq n_\epsilon(2\lambda_1, 2\lambda_2)+ n_\epsilon(2\lambda'_1,2\lambda'_2)= q_\epsilon(2L)+q_\epsilon(2L')\,,
\]
which shows that $q_\epsilon$ is convex.

As for (c), note that $t\mapsto q_\epsilon((1+tA)L)$ is a smooth function near $t=0$.
For $\epsilon>0$ this is clear, because $q_\epsilon$ is smooth.
For $\epsilon=0$ this depends on the fact that $L$ has rank $2$
and so has $(1+tA)L$ for $|t|$ small.
A straightforward computation using \eqref{eq:sh-ex} shows that
\begin{equation}\label{eq:pos}
   \frac{d\,}{dt}q_\epsilon((1+tA)L)\Big|_{t=0}=\frac{1}{2q_\epsilon(L)}\left( \tr(\hat{A})+\sqrt{\det(\epsilon^2\Id+L^*L)}\tr((\epsilon^2\Id+L^*L)^{-1}\hat{A})\right)~,
\end{equation}
where we have put
\[
  \hat{A}:=\frac{d\,}{dt}\Big(((\Id+tA)L)^*(\Id+tA)L\Big)\Big|_{t=0}=2L^*AL~.
\]
Now, $\hat{A}$ and $(\epsilon^2\Id+L^*L)$ are respectively non-negative and positive self-adjoint operators of $V$, so the  derivative in \eqref{eq:pos} is non-negative.
Finally, if $A\circ L\neq 0$, then $\hat{A}\neq 0$, and this implies the strict positivity of the derivative.
\end{proof}

After such preparation, we can now prove the main statement of this section.

\begin{proof}[Proof of Proposition \ref{lm:convexity}]
For $\epsilon\geq 0$ let $u_\epsilon(s)=q_\epsilon(T(s))$,
so that $u_0=u$.
We remark that for $\epsilon>0$ the function $u_\epsilon$ is smooth at every $s\in[0,1]$, whereas for $\epsilon=0$ it is smooth at those $s\in[0,1]$ such that
$T(s)$ has rank $2$.
At those points we have
\[
\ddot u_\epsilon (s)= (d^2 q_\epsilon)_{T(s)}(\dot T(s), \dot T(s))+(dq_\epsilon)_{T(s)}(\ddot T(s))
\]
where $(d^2 q_\epsilon):V\times V\rar W$ is the Hessian of $q_\epsilon$
and
we view $\ddot{u}_\epsilon(s)$, $\dot{T}(s)$ and $\ddot{T}(s)$ as elements of $\Hom(V,W)$.

As $q_\epsilon$ is convex, $(d^2 q_\epsilon)_{T(s)}(\dot T(s), \dot T(s))\geq 0$.
Moreover, $(dq_\epsilon)_{T(s)}(\ddot T(s))=\frac{d}{dt}q_\epsilon
\Big((1+tA(s))T(s)\Big)\Big|_{t=0}$. Hence,
by Lemma \ref{lm:sh-pr}(c) the term $(dq_\epsilon)_{T(s)}(\ddot T(s))$ is non-negative, and in fact strictly positive  if $A(s)\circ T(s)\neq 0$.

For $\epsilon=0$ this shows the positivity of $\ddot u$ at those $s\in[0,1]$ such that the rank of $T(s)$ is $2$ and $A(s)\circ T(s)\neq 0$.
On the other hand, for $\epsilon>0$ the function $u_\epsilon$ turns out to be convex.  The convexity of $u$ follows, since $u=\lim_{\epsilon\to 0}u_\epsilon$.
\end{proof}




\bibliographystyle{amsplain}
\bibliography{biblio}

\end{document}